\documentclass{amsart}
\usepackage{url}
\usepackage{lscape}
\usepackage{mathtools}
\usepackage[colorlinks=true,linkcolor=red,citecolor=blue]{hyperref}
\usepackage{color}
\usepackage[color,matrix,arrow, curve]{xy}
\usepackage{amsmath, amsthm, amsfonts, amssymb}
\usepackage{mathrsfs}
\usepackage{tikz}
\usetikzlibrary{matrix,arrows}

\DeclareMathOperator{\Hom}{Hom}

\DeclareMathOperator{\map}{Map}

\DeclareMathOperator{\End}{End}

\numberwithin{equation}{section}
\newtheorem{thm}{Theorem}[section]
\newtheorem*{thma}{Theorem A}
\newtheorem*{thma'}{Theorem A'}
\newtheorem*{thmb}{Theorem B}
\newtheorem*{thmb'}{Theorem B'}
\newtheorem*{thmc}{Theorem C}
\newtheorem*{thmd}{Theorem D}
\newtheorem*{thme}{Theorem E}
\newtheorem*{thme'}{Theorem E'}
\newtheorem{cor}[thm]{Corollary}
\newtheorem*{corA}{Corollary A}
\newtheorem{lem}[thm]{Lemma}

\newtheorem{prop}[thm]{Proposition}

\theoremstyle{definition}
\newtheorem{defn}[thm]{Definition}

\newtheorem{rem}[thm]{Remark}
\newtheorem{expl}[thm]{Example}

\newcommand{\id}{\textup{id}}

\newcommand{\aut}{\textup{Aut}}

\newcommand{\emb}{\textup{emb}}
\newcommand{\config}{\mathsf{con}} 

\newcommand{\fin}{\mathsf{Fin}}

\newcommand{\boxfin}{\mathsf{Boxfin}}

\newcommand{\simp}{\mathsf{simp}}
\newcommand{\op}{\textup{op}}

\newcommand{\FF}{\mathbb F}
\newcommand{\RR}{\mathbb R}
\newcommand{\ZZ}{\mathbb Z}

\newcommand{\QQ}{\mathbb Q}

\newcommand{\kk}{\Bbbk}

\newcommand{\PaB}{\mathrm{PaB}}

\newcommand{\Mat}{\mathrm{Mat}}

\newcommand{\limsub}[1]{\begin{array}[t]{cc} \textup{lim} \\
[-1.7mm] \scriptstyle{#1} \end{array}}

\newcommand{\colimsub}[1]{\begin{array}[t]{cc} \textup{colim} \\
[-1.7mm] \scriptstyle{#1} \end{array}}

\newcommand{\holimsub}[1]{\begin{array}[t]{cc} \textup{holim} \\ [-1mm]
\scriptstyle{#1} \end{array}}

\newcommand{\hocolimsub}[1]{\begin{array}[t]{cc} \textup{hocolim} \\
[-1.2mm] \scriptstyle{#1} \end{array}}

\newcommand{\uli}{\underline}

\newcommand{\pre}{\textup{pre}}

\newcommand{\fun}{\textup{Fun}}
\newcommand{\cat}{\mathbf}

\newcommand{\tree}{\mathsf{Tree}}

\renewcommand{\S}{\mathbf{S}}
\newcommand{\Op}{\mathbf{Op}}

\newcommand{\Pro}{\mathrm{Pro}}
\newcommand{\ProS}{\Pro(\S)}
\newcommand{\GT}{\mathrm{GT}}

\begin{document}

\title[Formality of little disks operads in positive characteristic]{On the Formality of the little disks operad in positive characteristic}
\author[Pedro Boavida]{Pedro Boavida de Brito}
\author{Geoffroy Horel}

\address{Dept. of Mathematics, IST, Univ. of Lisbon, Av. Rovisco Pais, Lisboa, Portugal}%
\email{pedrobbrito@tecnico.ulisboa.pt}

\address{Universit\'{e} Sorbonne Paris Nord, LAGA, CNRS, UMR 7539, 99 avenue Jean-Baptiste Cl\'{e}ment\\ 93430 Villetaneuse\\ France}
\address{\'{E}cole normale sup\'{e}rieure, DMA, CNRS, UMR 8553, 45 rue d'Ulm, 75230 Paris Cedex 05, France}
\email{horel@math.univ-paris13.fr}

\thanks{We gratefully acknowledge the support through: grant SFRH/BPD/99841/2014 and project MAT-PUR/31089/2017, funded by Funda\c c\~{a}o para a Ci\^{e}ncia e Tecnologia; projects ANR-14-CE25-0008 SAT, ANR-16-CE40-0003 ChroK, ANR-18-CE40-0017 PerGAMo, funded by Agence Nationale pour la Recherche. We also thank the Hausdorff Institute for Mathematics for the excellent working conditions during the Junior Trimester Program, where this project started. 
}
\begin{abstract}
Using a variant of the Boardman-Vogt tensor product, we construct an action of the Grothendieck-Teichm\"uller group on the completion of the little $n$-disks operad $E_n$. This action is used to establish a partial formality theorem for $E_n$ with mod $p$ coefficients and to give a new proof of the formality theorem in characteristic zero.
\end{abstract}
\maketitle

The quasi-isomorphism type of the singular chains of the little disks operad $E_n$ is by now well understood in characteristic zero thanks to the celebrated formality theorem. This theorem asserts that this dg-operad is \emph{formal}, i.e. it is quasi-isomorphic to its homology, a dg-operad with zero differential. The history of this result is long and winding. The case $n=2$ was first proved by Tamarkin in \cite{tamarkinformality}, the higher dimensional cases with coefficients in $\RR$ by Kontsevich in \cite{kontsevichoperads} and Lambrechts-Voli\'c in \cite{lambrechtsformality}. Finally, the case of $\QQ$-coefficients was done by Fresse-Willwacher in \cite{fresseintrinsic}.

In the present paper, we study the analogous question in positive characteristic. The case of the little $2$-disks operad was considered by Joana Cirici and the second author in \cite[Theorem 6.7]{CiriciHorel}. Given a prime $p$, our first  result is the following. 

\begin{thma}
The dg-operad $C_*(E_n, \FF_p)$ is $(n-1)(p-2)$-formal.
\end{thma}
Recall that a dg-operad $P$ (that is, an operad in chain complexes over a field) is called \emph{$N$-formal} if $P$ is connected to $H_*(P)$ via a zig-zag of maps of operads that induce isomorphisms in homology in degree $\leq N$. 

\medskip
For an integer $k$, let $E_n^{\leq k}$ be the truncation of $E_n$ to arities $\leq k$. This is the collection of spaces $E_n(i)$ with $i\leq k$ equipped with all the operations of the operad $E_n$ that only involve these spaces. Since the top non-vanishing degree of $H_*(E_n(k),\FF_p)$ is $(k-1)(n-1)$, we have:
\begin{corA}
The $(p-1)$-truncated dg-operad $C_*(E_n^ {\leq p-1},\FF_p)$ is formal.
\end{corA}

This statement is sharp. Indeed, as explained in \cite[Remark 6.9]{CiriciHorel}, the two chain complexes $C_*(E_n(p),\FF_p)$ and $H_*(E_n(p),\FF_p)$ are not quasi-isomorphic as $\Sigma_p$-equivariant chain complexes since their homotopy orbits differ. It follows that $C_*(E_n^{\leq p},\FF_p)$ is \emph{not} formal as a $p$-truncated dg-operad.

In relation to non-formality, we should also mention the recent results \cite{salvatoreplanar} of Paolo Salvatore who showed that the \emph{planar} (i.e. non-symmetric) little 2-disks operad is \emph{not} formal in characteristic $2$. He proved this by an explicit computation of the first obstruction to formality. Evidently, N-formality of operads implies N-formality of the underlying planar operads, so our results give a lower bound for potential obstructions to exist. It is unclear to us at this point if Salvatore's obstruction is specific to characteristic $2$ or if planar formality should never be expected to hold.

\medskip
The arguments used to prove Theorem A, which we will briefly outline at the end of the introduction, also produce a new proof of the formality theorem in characteristic zero. In fact, without much more effort we also obtain a relative version:
\begin{thmb}
Let $n$ and $d$ be positive integers and $d\geq 2$. Let $E_n\to E_{n+d}$ be a map between little disks operads. Then
\begin{enumerate}
\item[($\QQ$)] the map of dg-operads $C_*(E_n,\QQ)\to C_*(E_{n+d},\QQ)$ is formal,
\end{enumerate}
and
\begin{enumerate}
\item[(${\FF_p}$)] the map of dg-operads $C_*(E_n,\FF_p)\to C_*(E_{n+d},\FF_p)$ is $N$-formal, 
\end{enumerate}
where $N = \Delta(p-2)$ and $\Delta$ is the greatest common divisor of $n-1$ and $d$.
\end{thmb}

The characteristic zero case of this theorem was established in \cite{lambrechtsformality} with a weaker range and, with the range as above, in \cite{turchinrelative} and \cite{fresseintrinsic}. Our proof in the rational case parallels the positive characteristic one and is fundamentally different from the proofs in these papers.

\medskip
Before we explain the main ideas in this paper, let us mention two other related results that we prove. The first one is a new proof of the formality of $E_n$ as a Hopf cooperad with rational coefficients. Denoting by $\Omega^*_{poly}$ Sullivan's functor of polynomial differential forms on spaces, we obtain:

\begin{thmc}
For all $n$, the $\infty$-Hopf cooperad $\Omega_{poly}^*(E_n)$ is quasi-isomorphic to $H^*(E_n,\QQ)$ as an $\infty$-Hopf cooperad.
\end{thmc}

An $\infty$-Hopf cooperad is a homotopically coherent version of a strict Hopf cooperad, i.e. a cooperad in commutative differential graded algebras. Their appearance in this theorem stems from the fact that $\Omega^*_{poly}$ is not a lax monoidal functor and therefore $\Omega^*_{poly}(E_n)$ is not a strict Hopf cooperad.

Finally, we also deduce partial formality for configuration spaces of points in $\RR^n$ with coefficients in $\FF_p$.

\begin{thmd}
For $d\geq 3$, the dg-algebra $C^*(\mathrm{Conf}_n(\RR^d),\FF_p)$ is $(d-1)(p-2)$-formal.
\end{thmd}

Two related formality results can be found in \cite[Theorem 3.10]{salvatorenon}. On the one hand, Salvatore shows that the space $\mathrm{Conf}_n(\RR^d)$ is formal with any ring of coefficients if $d\geq n$ (and in fact is even intrinsically formal), on the other hand he shows that $C^*(\mathrm{Conf}_n(\RR^2),\FF_2)$ is not a formal dg-algebra if $n\geq 4$.

\subsection*{Our approach} The first step in our strategy to prove theorems A and B is to endow the $p$-completion of the little $n$-disks operad with a non-trivial action of the pro-$p$ version of the Grothendieck-Teichm\"uller group. This group is denoted $\GT_{\FF_p}$ in this paper. It is equipped with a surjective group homomorphism
\[\chi:\GT_{\FF_p}\to \ZZ_p^\times\]
called the cyclotomic character. The reason for this terminology is that there is a group homomorphism $\mathrm{Gal}(\overline{\QQ}/\QQ)\to \GT_{\FF_p}$ and the map $\chi$ factors the usual $p$-adic cyclotomic character $\mathrm{Gal}(\overline{\QQ}/\QQ)\to \ZZ_p^{\times}$. Our key result is the following.

\begin{thme}
There exists an action of $\GT_{\FF_p}$ on $E_n^{\wedge}$, the $p$-completion of $E_n$, such that the induced map
\[\GT_{\FF_p}\to\pi_0\mathrm{Aut}^h(E_n^{\wedge})\xrightarrow{\textup{restr.}}\pi_0\mathrm{Aut}^h(E_n^{\wedge}(2))\simeq\ZZ_p^{\times}\]
coincides with the cyclotomic character.
\end{thme}

By $\aut^h$ we mean the space of homotopy automorphisms. We also have a rational version of this theorem. In that case, one has to use the pro-unipotent Grothendieck-Teichm\"uller group, denoted $\GT_{\QQ}$. This group was initially introduced by Drinfeld in \cite{drinfeldquasi} and can be viewed as the group of automorphisms of the rationalization of the little $2$-disks operad (see Definition \ref{defn : GT} for a precise definition). As in the pro-$p$-case, there is a cyclotomic character map $\GT_{\QQ}\to\QQ^{\times}$.

\begin{thme'}
There exists an action of $\GT_{\QQ}$ on $E_n^{\wedge}$, the $\QQ$-completion of $E_n$, such that the induced map
\[\GT_{\QQ}\to\pi_0\mathrm{Aut}^h(E_n^{\wedge})\xrightarrow{\textup{restr.}}\pi_0\mathrm{Aut}^h(E_n^{\wedge}(2))\simeq\QQ^{\times}\]
coincides with the cyclotomic character.
\end{thme'}

Now, to prove theorem A, we use the surjectivity of the cyclotomic character and of the restriction map $\ZZ_p^\times\to\FF_p^\times$. For any unit $u$ of $\FF_p$, we obtain an automorphism of $C_*(E_n,\FF_p)$ that acts by multiplication by $u^k$ in homological degree $(n-1)k$. If we pick $u$ of maximal order, we can use this automorphism in order to functorially split the Postnikov tower in a range and get a proof of Theorem A. In the rational case, we can lift a unit of $\QQ$ of infinite order and get a full splitting of the Postnikov tower. The other theorems indicated above are also proved by exploiting this action of the Grothendieck-Teichm\"uller group. 

Let us mention that, although Theorem E and E' look extremely similar, the formality results that we get in the case of $\FF_p$-coefficients are weaker than for $\QQ$-coefficients. The reason is that the field $\FF_p$ does not contain units of inifinite order contrary to the field $\QQ$. Note however that this difference is not simply a feature of the method used in this paper but is intrinsic to the situation. Indeed, as we have said above, the formality result that we prove with $\FF_p$-coefficients is sharp.

The general method to deduce formality from pure actions is an elaboration of an idea that seems to date back to Sullivan, for commutative differential graded algebras in characteristic zero. This was promoted to a statement about operads by Guill\'{e}n-Navarro-Pascual-Roig \cite{GNPR}, and used by Petersen \cite{Petersen} to give a proof of rational formality for $E_2$ using its $\GT_{\QQ}$ action. Our proof follows this strategy, adapted to the positive characteristic using work of Joana Cirici and the second author \cite{CiriciHorel}.

In order to construct the said action on the completion of $E_n$, we use a version of the Dunn-Lurie additivity theorem, established by Michael Weiss and the first author in \cite{BoavidaWeissBV}. These theorems express $E_{n+m}$ as a suitable tensor product of $E_n$ and $E_m$. As in \emph{loc. cit}, we make use of the notion of configuration category of $\RR^n$, an object closely related to the little disks operad $E_n$. Our approach to prove theorems $E$ and $E^\prime$ is to use a completed version of the additivity theorem, together with the action of Grothendieck-Teichm\"uller group on the completion of $E_2$ defined by Drinfeld. The desired action of $\textup{GT}$ on the completion of $E_n$ is then obtained by tensoring the completion of $E_2$ (with Drinfeld's action) with the completion of $E_{n-2}$ (with the trivial action).

\medskip
The characteristic zero formality theorem has had several important applications. One is Tamarkin's proof of Kontsevich's formality theorem \cite{kontsevichoperads, hinichtamarkin}. This states that the Hochschild cochains of the ring of smooth functions on a manifolds is quasi-isomorphic to its cohomology (the algebra of polyvector fields over the manifold) as algebras over the little $2$-disks operad. We hope that our formality theorem can be used to prove a positive characteristic version of this theorem and its higher dimensional versions as in \cite{calaquetriviality}.

Another beautiful application of the (relative) formality of the higher dimensional little disks operads is the proof of the rational collapse of the Vassiliev spectral sequence due to Lambrechts, Turchin and Voli\'c \cite{lambrechtslongknots}. This spectral sequence computes the homology of the space of long knots in $\RR^{1+d}$, that is, the space of smooth embeddings $\RR\to\RR^{1+d}$ that coincide with the standard inclusion outside of a compact subspace. In higher dimensions, Arone and Turchin \cite{aroneturchin} also used the formality of the map $E_n\to E_{n+d}$ in order to produce graph complexes computing the rational homology of the space of embeddings $\RR^n\to\RR^{n+d}$ that are standard outside of a compact subspace. These results have been extended by Fresse-Turchin-Willwacher \cite{FTW}. In \cite{Galoisknots}, we explore some consequences of Theorem B to embedding calculus and in particular to the Goodwillie-Weiss-Vassiliev spectral sequence, beyond the rational context.

\subsection*{Acknowledgements} We wish to thank Joana Cirici, Pascal Lambrechts, Thomas Nikolaus, Dan Petersen, Alex Suciu, Victor Turchin, and Peter Teichner for helpful conversations.

\section{Configuration categories}\label{sec:confcats}
We denote by $\S$ the category of simplicial sets. Given $k \geq 0$, we write $\uli k$ for $\{1, \dots , k\}$. (By convention, $\uli 0$ is the empty set.)

\begin{defn}
Let $M$ be a topological manifold. The \emph{configuration category} of $M$ is a topological category $\config(M)$. Its space of objects is the disjoint union of the spaces of ordered configurations of $k$ points in $M$ for $k \geq 0$. A morphism from $x : \uli k \hookrightarrow M$ to $y : \uli \ell \hookrightarrow M$ consists of a map of finite sets $f : \uli k \to \uli \ell$ and a Moore path $h : [0,a] \to M^{\times k}$ satisfying:
\begin{itemize}
\item $h(0) = x$, $h(a) = y \circ f$, 
\item if two components of $h(T)$ agree for some $T$, then they agree for all $t \geq T$.
\end{itemize}
\end{defn}

This definition goes back to Ricardo Andrade's thesis \cite{Andrade}, if not earlier. Different models for configuration categories were given in \cite{BoavidaWeissLong}, where they were studied in connection with spaces of smooth embeddings. In \emph{loc. cit.}, a  close relationship between the little disks operad $E_n$ and the configuration category of Euclidean space $\RR^n$ was explained. We will review this below.

\medskip
Let $\fin$ be the skeleton of the category of finite sets with objects $\uli k$, for $k \geq 0$.  There is a forgetful map $\config(M) \to \fin$. Taking the usual nerve $N$, we obtain a map of simplicial spaces 
\[
N \config(M) \to N \fin
\]
which is a fibration between Segal spaces (i.e. a fibration in the complete Segal space model structure).

We will often take an alternative but equivalent point of view on configuration categories, which we now describe. Recall that for a simplicial set $A$ the category of simplices of $A$, denoted $\simp(A)$, is the category whose objects are pairs $([n], x)$ with $n \geq 0$ and $x \in A$; the set of morphisms $([n], x) \to ([m],y)$ is the set of morphisms $\theta : [n] \to [m]$ in $\Delta$ such that $x = \theta^*y$. We will use this construction when $A = NC$ for a category $C$ and we simply write $\simp(C)$ for $\simp(NC)$. 

It is easy to see that a simplicial space over $A$ is the same data as a contravariant functor from $\simp(A)$ to $\S$. Given such a simplicial map $p : X \to A$, a functor is produced by the assigning to $x \in \simp(A)$ the fiber of $p$ at $x$. This allows us to think of the configuration category $\config(M)$ with its reference map to $\fin$ as a functor
\[
\simp(\fin)^\op \to \S \;
\]
which we still denote by $\config(M)$.

\subsection{Relation to little disks operads}\label{sec:littledisks}

Recall the category $\tree$ of non-empty, finite rooted trees of \cite{MoerdijkWeiss} (Moerdijk-Weiss write $\Omega$ for $\tree$). Important objects in $\tree$ are the tree $C_n$ with a single vertex and $n+1$ edges ($n$ of which are called leaves), for $n \geq 0$, and the tree $\eta$ with a single edge and no vertices. Any other object in $\tree$ may be written as an union of objects of the form $C_n$, intersecting along $\eta$'s. 

Following Cisinski-Moerdijk (see \cite[Definition 5.4]{cisinskidendroidal}), we make the following definition.
\begin{defn}
A (monochromatic) $\infty$-operad in spaces is a functor $$X : \tree^{\op} \to \S$$ satisfying the following two conditions:
\begin{enumerate}
\item the space $X_{\eta}$ is contractible,
\item the Segal map
\[
X_T \to \prod_{v \in T} X_{C_{|v|}}
\]
is a weak equivalence. Here, the product runs over the vertices $v \in T$ and $C_{|v|}$ is the corolla whose set of leaves is the set of inputs at $v$.

\end{enumerate}
\end{defn}

We denote by $\cat{OpS}$ the category of operads in spaces. Let us briefly recall how operads are related to contravariant functors on $\tree$. To any tree $T$, one can associate an operad $F(T)$. If $T=C_n$, $F(T)$ is simply the free operad on an operation of arity $n$. If $T$ is a arbitrary tree, $F(T)$ is defined as the coproduct (in the category of operads in sets)
\[F(T)=\sqcup_{v\in T}F(C_{|v|}) \; .\]
The assignment $T\mapsto F(T)$ is functorial. The dendroidal nerve functor
\[N_d:\cat{OpS}\to\fun(\tree^\op,\S)\]
sends an operad $P$ to the functor $N_d(P)$ defined by the formula
\[N_d(P)_T=\map(F(T),P).\]
Clearly, $N_d(P)$ satisfies conditions (1) and (2) above. Moreover, it sends weak equivalences of operads to objectwise weak equivalences of functors, and so it induces a map on homotopy function complexes
\begin{equation}\label{eq:Nd}
\map^h(P,Q) \to \map^h(N_d P, N_d Q)
\end{equation}
for any two operads $P$ and $Q$.

\begin{thm}[Cisinski--Moerdijk  \cite{cisinskidendroidal}]\label{thm: cm}
The functor $N_d$ induces an equivalence between the $\infty$-category of operads in spaces and the $\infty$-category of $\infty$-operads.
\end{thm}

In other words, the map (\ref{eq:Nd}) is a weak equivalence for any two operads $P$ and $Q$, and $N_d$ is homotopically essentially surjective. In view of this theorem, we will often not make a distinction between an operad $P$ and its dendroidal nerve $N_d P$.

\medskip
We denote by $\Op_{\infty}\S$ the category of $\infty$-operads. The link between configuration categories and $\infty$-operads goes via an inclusion $j : \simp(\fin) \to \tree$ given by regarding a string of maps of finite sets as a levelled tree (adding a root, and capping off each leaf with a vertex). This induces a functor
\[
j^* : \fun(\tree^\op, \S) \to \fun(\simp(\fin)^\op, \S)
\]
by restriction. For the little $n$-disks operad $E_n$, it is not difficult to identify up to weak equivalence the objects $j^* E_n$ and $\config(\RR^n)$. Less trivial is the observation that $E_n$ can be reconstructed from $\config(\RR^n)$:

\begin{thm}[{\cite[Thm 7.5]{BoavidaWeissLong}}]\label{thm: from simpfin to omega}
Let $\cat{R}$ denote the subcategory of the category of $\infty$-operads $X$ such that $X_{C_0}$ and $X_{C_1}$ are weakly contractible. Let $K$ denote the homotopy left adjoint to the inclusion of $\cat{R}$ in $\Op_{\infty}\S$. Then the natural maps of $\infty$-operads
\[
E_n \to K(E_n) \gets  K(j_! \config(\RR^n))
\]
are weak equivalences, where $j_!$ denotes the homotopy left adjoint to $j^*$.
\end{thm}

Equivalently, this theorem says that the derived counit map
\[
j_! j^* E_n \to E_n
\]
is a weak equivalence in a model structure on the category of dendroidal spaces whose fibrant objects are those dendroidal spaces that belong to $\cat{R}$. 

\begin{cor}\label{cor:aut} The map
\[
j^* : \aut^h(E_n) \to \aut^h(\config(\RR^n))
\] is a weak equivalence.
\end{cor}

\begin{rem}
The configuration category of $\RR^n$ is a \emph{part} of the little disks operad $E_n$, from which the little disks operad can be reconstructed. Later on, we use a variant of the Boardman-Vogt tensor product, the $\boxtimes$ product. This is a product between functors on $\simp(\fin)$, but it can be extended formally to a product on $\tree$, i.e. to a product between $\infty$-operads. For the little disks operad, the two products $\boxtimes$ and $\otimes_{BV}$ agree (since additivity holds in both contexts), and conjecturally they agree for all operads satisfying the condition of the theorem above.

Having said that, the reader may wonder why we consider configuration categories, and why we did not just write this paper in operadic language. There are a few reasons, some more circumstantial than others. One is that a theory for $\infty$-operads (and corresponding tensor product) in general cartesian $\infty$-categories (like the $\infty$-category of profinite spaces) is not available. If such framework existed, this paper could probably be written in that language. But that would essentially be a translation. And given that the configuration category of $\RR^n$ is somewhat smaller than the little disks operad, some separate analysis would presumably be needed. Another reason is, of course, the possibility of extending some results to more general manifolds.
\end{rem}

\section{Localizations and completions}

In this section, we review notions of completion and localization from the point of view of pro-spaces. We do so in order to treat the two notions of completion ($p$-completion and rational completion) in a uniform way. Standard references are \cite{ArtinMazur}, \cite{BK} and \cite{Sullivan}.

\medskip
Recall that the pro-category $\ProS$ is the category whose objects are functors $I \to \S$ where $I$ is a varying small cofiltered category. The set of morphisms from $X : I \to \S$ to $Y : J \to \S$ is given by
\[
\limsub{j \in J} \colimsub{i\in I} \Hom_\S(X_i, Y_j) \; .
\]
(An alternative definition: the opposite of $\ProS$ is the category of set-valued functors on $\S$ which are cofiltered limits of representables. This is also the category of set-valued functors on $\S$ that preserve finite limits.)

Any map $X \to Y$ of pro-spaces is isomorphic in $\ProS$ to a natural transformation $X^\prime \to Y^\prime$ (so in particular $X^\prime$ and $Y^\prime$ have the same indexing category). Such a map is called \emph{strict}. A map of pro-spaces is said to be a strict \emph{weak equivalence} if it is isomorphic to a strict map which is an objectwise weak equivalence.

The category $\ProS$ with the notion of strict weak equivalences has derived mapping spaces denoted 
\[
\map^h_{\ProS}(-,-)
\]
or simply $\map^h(-,-)$. Given pro-spaces $X : I \to \S$ and $Y : J \to \S$, there is an identification
\begin{equation}\label{eq:pro-formula}
\map^h_{\ProS}(X,Y) \simeq \holimsub{j \in J} \hocolimsub{i \in I} \map^h_{\S}(X_i, Y_j) \;
\end{equation}
where $\map^h$ on the right-hand side denotes the derived mapping space in $\S$. Finally, any space $K$ may be viewed as a pro-space over the one-point category.

\begin{defn}
Let $C$ be a small full subcategory of $\S$. Let $C^{cofin}$ be the closure of $C$ under finite homotopy limits and retracts. We declare a map of pro-spaces $X\to Y$ to be a \emph{$C$-equivalence} if the induced map
\[
\map^h(Y,K)\to \map^h(X,K)
\]
is a weak equivalence for all $K$ in $C$ (and therefore also for all $K$ in $C^{cofin}$).
\end{defn}

Given a pro-space $X : I \to \S$ and $A$ an abelian group, we will write $H^n(X, A)$ for the continuous cohomology of $X$ with coefficients in $A$, i.e.
\[
H^n(X,A):=\colimsub{i \in I^{\op}} H^n(X_i, A)\cong\pi_0 \map^h(X,K(A,n))
\]
where $H^n(X_i,A)$ denotes the usual singular cohomology.

\medskip
Isaksen shows in \cite{isaksencompletions} that there exists a model structure $\ProS_C$ on $\Pro(\S)$ where the weak equivalences are the $C$-equivalences. As with any localization, the identity functor
\[\ProS_C\to\ProS\]
is a right Quillen functor that induces a full and faithful inclusion between the underlying $\infty$-categories, i.e. induces a weak equivalence on derived mapping spaces. Its essential image consists of  $C$-local pro-spaces, in the following sense. We say that a pro-space $W$ is $C$-\emph{local} if for every $C$-equivalence $X \to Y$, the restriction map
\[
\map^h(Y,W)\to \map^h(X,W)
\]
is a weak equivalence.

\begin{prop}
A pro-space is $C$-local if and only if it is a homotopy limit, computed in $\ProS$, of a cofiltered diagram $I \to C^{cofin} \to \ProS$.
\end{prop}

\begin{proof}
By construction, any space in $C^{cofin}$ is $C$-local. The homotopy limit of a diagram of $C$-local pro-spaces is $C$-local. This settles one implication.

Conversely, Christensen and Isaksen prove in \cite[Proposition 4.9]{christensenisaksen} that any fibrant object of $\ProS_C$ is isomorphic to a pro-space which is levelwise $C$-nilpotent. A $C$-nilpotent space is the limit of a finite tower of principal fibrations with fibers in $C$ (see \cite[Definition 3.1]{isaksencompletions}). Such a space is obviously in $C^{cofin}$. Therefore, every $C$-local pro-space is weakly equivalent in $\ProS$ to a pro-space which is levelwise in $C^{cofin}$. By formula (\ref{eq:pro-formula}), every pro-space $X : I \to \S$ is the homotopy limit in $\ProS$ of the cofiltered diagram of constant pro-spaces $X_i$, so the claim follows.
\end{proof}

For an $\infty$-category $C$, we denote by $\Pro_{\infty}(C)$ its pro-category in the $\infty$-categorical sense (as defined in \cite[Definition 5.3.5.1]{lurietopos}). (Briefly, and under smallness hypothesis on $C$ which we suppress, a pro-object in $C$ is a left fibration $U \to C$ where $U$ is a cofiltered $\infty$-category.) We use the same notation for a model category and the $\infty$-category it presents.

\begin{thm}[{\cite[Theorem 7.1.2]{bhh}}]\label{theo : interpretation of proSC}
The $\infty$-category underlying $\ProS_C$ is equivalent to $\Pro_{\infty}(C^{cofin})$, the pro-category of the $\infty$-category $C^{cofin}$.
\end{thm}
\begin{proof}
We give a sketch of the proof and we take the opportunity to correct a mistake in the proof of \cite[Theorem 7.1.2]{bhh}. It is claimed there that what Isaksen calls $C$-nilpotent spaces is the closure of $C$ under finite homotopy limits. This is not how Isaksen defines $C$-nilpotent spaces (see \cite[Definition 3.1]{isaksencompletions}): $C$-nilpotent spaces in the sense of Isaksen are contained in the closure of $C$ under homotopy limits (which is the only thing that is used in the proof of \cite[Theorem 7.1.2]{bhh}) but the converse is probably not true in general.

Since $C^{cofin}$ is a full subcategory of $\S$, $\Pro_{\infty}(C^{cofin})$ is a full $\infty$-subcategory of $\Pro_{\infty}(\S)$. Moreover, $\ProS_C$ is a full $\infty$-subcategory of $\ProS$ and there is a comparison 
\[
\ProS \to \Pro_{\infty}(\S)
\]
which is an equivalence of $\infty$-categories by \cite[Theorem 5.2.1]{bhh}. Therefore, the claim is that the essential images of $\Pro_{\infty}(C^{cofin})$ and $\ProS_C$ in $\Pro_{\infty}(\S)$ agree. This follows from the proposition above.
\end{proof}

\subsection{Pro $p$-completion and rational completion}

\begin{prop}\label{prop : characterization of F_p cofin}
Let $C$ be the set of spaces $K(\ZZ/p, n)$ for $n \geq 0$. Then $C^{cofin}$ is the category of truncated spaces with finitely many components and whose homotopy groups are all finite $p$-groups.
\end{prop}

\begin{proof}
It is clear that a space in $C^{cofin}$ must satisfy these conditions.

Conversely, $C^{cofin}$ contains all the spaces $K((\ZZ/p)^n,0)$ for $n$ a positive integer. In particular, $C^{cofin}$ contains all finite sets as any finite set can be written as a retract of $(\ZZ/p)^n$ for $n$ large enough. If $X$ is the disjoint union of $A$ and $B$ then $X$ is a retract of $A \times B \times S^0$, and so $X$ is in $C^{cofin}$ if $A$ and $B$ are.

Now, a connected space $X$ satisfying the conditions of the Proposition is necessarily nilpotent (finite $p$-groups $G$ are nilpotent and every finite abelian $p$-group with an action of a finite $p$-group is nilpotent). Therefore, the Postnikov tower of $X \in C^{cofin}$ may be refined to a finite tower of principal fibrations where at each step we perform a homotopy base change along a map $* \to K(A, \ell)$ for $A$ a finite abelian $p$-group \cite[VI.6.1]{GJ}. Since $K(A, \ell)$ is in $C^{cofin}$, we are done.
\end{proof}

\begin{defn}\label{def : unipotent}
A $\QQ$-\emph{unipotent} group $G$ is a group which is uniquely divisible, nilpotent and such that $H_1(G, \QQ)$ is finite-dimensional.  A representation of a group $G$ on a $\QQ$-vector space is said to be $\QQ$-\emph{unipotent} if it admits a filtration preserved by the $G$-action and whose graded pieces are trivial representations.  
\end{defn}

Note that, for a $\QQ$-unipotent group, each graded piece of the lower central series must be a finite-dimensional $\QQ$-vector space. In particular, an abelian group is $\QQ$-unipotent if and only if it is a finite dimensional vector space over $\QQ$. 

\begin{prop}\label{prop : characterization of Q cofin}
Let $C$ be the set of spaces of the form $K(\QQ,n)$ for $n \geq 0$. Let $X$ be a space. Assume that the following conditions are satisfied:
\begin{enumerate}
\item the space $X$ is connected and truncated,
\item the homotopy groups $\pi_n X$ are $\QQ$-unipotent for $n \geq 1$,
\item the action of $\pi_1 X$ on $\pi_n X$ is $\QQ$-unipotent,
\end{enumerate}
then $X$ is in $C^{cofin}$.

Conversely, if $X$ is in $C^{cofin}$ then any path component of $X$ satisfies the conditions (1), (2) and (3) above.
\end{prop}

\begin{proof}
If the action of $\pi_1 X$ on $\pi_n X$ is $\QQ$-unipotent, then it is in particular nilpotent. Therefore, the Postnikov tower of $X$ can be refined to a tower of principal fibrations. Then the proof of the previous proposition applies \emph{mutatis mutandi}.

In order to prove the converse, it suffices to prove that the spaces whose components satisfy (1) (2) and (3) above contain all the spaces $K(\QQ,n)$ and are stable under finite homotopy limits and retracts. Is is obvious that this class of spaces contains all the spaces $K(\QQ,n)$ and that is stable under retracts. This class of spaces also contains the terminal space. Finally, consider a homotopy cartesian square
\[
\xymatrix{
X\ar[r]\ar[d]& E\ar[d]\\
A\ar[r]&B
}
\]
in which all components of $E$, $B$ and $A$ satisfy (1), (2) and (3). Clearly each component of $X$ will satisfy condition (1). On the other hand, by \cite[Proposition 4.4.3]{mayponto}, each component of $X$ is a $\mathscr{C}$-nilpotent space (see \cite[Definition 3.1.4]{mayponto} with respect to the abelian category $\mathscr{C}$ of finite dimensional $\QQ$-vector spaces. This means that conditions (2) and (3) are satisfied.
\end{proof}

We write 
\[
\S_{\FF_p}^{cofin} \quad \textup{ and } \quad \S_{\QQ}^{cofin}
\]
for the categories $C^{cofin}$ of the two propositions above. We denote the corresponding localizations by $\ProS_{\FF_p}$ and $\ProS_{\QQ}$ respectively.

\begin{defn}\label{defn:materialization}
For a pro-space $X$, we let $\Mat(X)$ denote the space obtained by taking the limit of the diagram $X$. This is the right adjoint to the functor obtained by taking the constant pro-space on a given space.
\end{defn}

Here we follow Lurie's terminology and notation in \cite{luriehigher} and refer to $\Mat(X)$ as the \emph{materialization} of $X$.

\begin{defn}\label{defn:completion}
Let $Y$ be a space (or more generally a pro-space). Let $\kk$ be $\FF_p$ or $\QQ$. The $\kk$-\emph{completion} of $Y$ is a fibrant replacement of $Y$ in $\ProS_\kk$. We denote this by $Y^\wedge$ (or $Y^\wedge_{\FF_p}$ or $Y^\wedge_{\QQ}$ if we need to be explicit).
\end{defn}

In $\infty$-categorical terms, using Theorem \ref{theo : interpretation of proSC}, completion is the left adjoint to the inclusion 
\[
\Pro_{\infty}(\S_{\kk}^{cofin}) \hookrightarrow \Pro_{\infty}(\S) \;.
\]
Informally, $Y^\wedge$ is given by the diagram on the comma category $Y/\S^{cofin}_{\kk}$ which takes an object $(Y \to Z)$ with $Z \in \S^{cofin}_{\kk}$ to $Z$.

\begin{prop}\label{prop : completion commutes with product}
Let $X$ and $Y$ be two spaces. The canonical map
\[{(X\times Y)}^\wedge\to X^\wedge \times {Y}^\wedge \]
is an equivalence in $\ProS_\kk$.
\end{prop}

\begin{proof}
By definition, a map of pro-spaces $X\to Y$ is a $\kk$-weak equivalence if the induced map $H^*(Y,\kk)\to H^*(X,\kk)$ is an isomorphism. We have to check that 
\[H^*({X}^\wedge \times {Y}^\wedge,\kk)\to H^*(({X\times Y})^\wedge,\kk)=H^*(X\times Y,\kk)\]
is an isomorphism. Let us write $X^\wedge=\{X_i\}_{i\in I}$ and $Y^\wedge=\{Y_j\}_{j\in J}$. Then
\[X^\wedge \times Y^\wedge\cong \{X_i\times Y_j\}_{i\in I,j\in J}\]
and, therefore, by the K\"unneth isomorphism we have
\[H^*(X^\wedge \times Y^\wedge,\kk)\cong \colimsub{I\times J} H^*(X_i,\kk)\otimes H^*(Y_j,\kk) \;\] 
which is $H^*(X,\kk)\otimes H^*(Y,\kk)$ by Fubini and the fact that the tensor product commutes with filtered colimits, 
\end{proof}

\begin{rem}
A word of warning: the terminology \emph{completion} has several, subtly different meanings in the literature. On the one hand, there is the $p$-completion as defined by Bousfield which is the localization of the category of spaces at the $\FF_p$-homology equivalences. On the other hand, there is the Bousfield-Kan $\FF_p$-completion that is given as the limit of an explicit cosimplicial diagram of spaces (see \cite{BK}). Finally, there is the $\FF_p$-completion as defined above. And there are rational variants of these three constructions.

Unlike the first two, the third object is a pro-space and not just a space. The next proposition says that, in some cases, the three notions of completions coincide in the category of spaces.
\end{rem}

\begin{prop}\label{prop: completion of nilpotent spaces}
Let $\kk=\FF_p$ or $\QQ$. Let $X$ be a nilpotent space of finite type (i.e. the homology of $X$ is finite-dimensional in each degree). Then the following three constructions define weakly equivalent spaces~:
\begin{enumerate}
\item the localization of $X$ with respect to the $\kk$-homology equivalences,
\item the Bousfield-Kan $\kk$-completion $\kk_\infty X$ of $X$,
\item the materialization $\RR\Mat(X^{\wedge}_\kk)$.
\end{enumerate}
\end{prop}

\begin{proof}
For $X$ of finite type, Isaksen \cite{isaksencompletions} proves that the Bousfield-Kan tower 
\[
X \to \{\kk_{s} X\}
\]
is an explicit fibrant replacement in $\ProS_\kk$. This shows that (2) and (3) are weakly equivalent. Moreover, nilpotent spaces of finite type are good in the sense of Bousfield-Kan for $\kk = \QQ$ \cite[V.3.1]{BK} and $\kk = \FF_p$ \cite[VI.5.3]{BK}, which means that the map from $X$ to the homotopy limit of $\{\kk_s X\}$ is an $\kk$-homology isomorphism. Finally, since the spaces $K(\kk,n)$ are $\kk$-local and since $\RR\Mat(X^{\wedge}_\kk)$ is constructed from spaces of this form using homotopy limits, we know that $\RR\Mat(X^{\wedge}_\kk)$ is $\kk$-local. It follows that the map $X\mapsto \RR\Mat(X^{\wedge}_\kk)$ is a $\kk$-localization.
\end{proof}

\section{Goodness}\label{sec:goodness}
Let $\mathbf{Gpd}$ be the category of groupoids. We consider subcategories
\[
\cat{Gpd}^{cofin}_{\kk} \subset \mathbf{Gpd}
\]
for $\kk=\FF_p$ and $\kk=\QQ$. These are the subcategories of groupoids having finitely many objects and automorphism groups a finite $p$-group or a $\QQ$-unipotent group, respectively. 

\begin{defn}\label{defn:good-group}
The $\kk$-\emph{completion} of a groupoid $G$ (or more generally a pro-groupoid) is the image of $G$ by the left adjoint to the inclusion 
\[\Pro(\cat{Gpd}^{cofin}_{\kk}) \to \Pro(\mathbf{Gpd}).\]
As before, we denote this by $G^{\wedge}_{\FF_p}$ or $G^\wedge_\QQ$, or simply $G^{\wedge}$ if the field is implicit.
\end{defn}

\begin{prop}
If $G \in \cat{Gpd}^{cofin}_{\kk}$ then $BG \in \S^{cofin}_{\kk}$.
\end{prop}
\begin{proof}
We prove the $\QQ$-unipotent case, the other case is easier and left to the reader. For simplicity, we will assume that $G$ is a $\QQ$-unipotent group. 

Suppose first that $G$ is abelian. Then $G$ is just a finite-dimensional $\QQ$-vector space. Therefore $BG$ is a finite product of copies of $K(\QQ,1)$, and so belongs to $\S^{cofin}_{\QQ}$.

For the general case, take the lower central series $\Gamma^k G$ for $G$. The fiber sequence
\[
B(\Gamma^k G/\Gamma^{k-1} G) \to B(G/\Gamma^k G) \to B(G/\Gamma^{k-1} G)
\]
extends one step to the right:
\[
B(G/\Gamma^k G) \to B(G/\Gamma^{k-1} G) \to K(\Gamma^k G/\Gamma^{k-1} G,2) \; .
\]
Indeed, $\Gamma^k G/\Gamma^{k-1} G$ is central in $G/\Gamma^k G$. This means that the conjugation action of 
$G/\Gamma^{k-1} G$ on $\Gamma^k G/\Gamma^{k-1} G$ is trivial, which implies that the first fiber sequence is principal. Now, since $G$ is $\QQ$-unipotent, $\Gamma^k G/\Gamma^{k-1} G$ is a finite-dimensional $\QQ$-vector space and therefore $K(\Gamma^k G/\Gamma^{k-1} G,2)$ belongs to $\S^{cofin}_{\QQ}$. By induction, we conclude that $B(G/\Gamma^k G)$ belongs to $\S^{cofin}_{\QQ}$ for all $k$.
\end{proof}

The canonical map $BG \to B(G^\wedge)$ extends, by the universal property of completion of pro-spaces, to a map
\begin{equation}\label{eq:goodmap}
(BG)^\wedge \to B(G^\wedge) \;
\end{equation}
of pro-spaces.

\begin{defn}\label{defi : goodness} We say that a groupoid $G$ is $\kk$-\emph{good} if $(\ref{eq:goodmap})$ is a weak equivalence in $\ProS_\kk$.
\end{defn}

Equivalently, a groupoid $G$ is $\kk$-good if the map $BG \to B(G^\wedge)$ is a weak equivalence in $\ProS_\kk$.

\begin{rem}
This (perhaps unfortunate) terminology in the case of profinite completion is due to Serre. Beware that this is different from what Bousfield-Kan call a good space. This should not lead to any confusion since we will never use Bousfield-Kan terminology in this paper.
\end{rem}

\begin{prop}\label{prop : fg abelian groups are good}
Finitely generated abelian groups are $\kk$-good. 
\end{prop}

\begin{proof}

By the previous proposition, good groups are stable under finite products so it suffices to prove that $\ZZ$ and $\ZZ/n$ are good. Then one checks easily that the $\QQ$-completion of $B\ZZ$ is $B\QQ$, the $\QQ$-completion of $B\ZZ/n$ is trivial; the $p$-completion of $B\ZZ$ is $\lim_nB\ZZ/p^n$ and the $p$-completion of $B\ZZ/n$ is $B\ZZ/p^m$ where $p^m$ is the largest power of $p$ that divides $n$.
\end{proof}

Another important example for us is the following.

\begin{prop}\label{prop : braid groups are good}
The pure braid groups $P_n$ are $\kk$-good. 
\end{prop}

\begin{proof}
The proof of the case $\kk=\QQ$ is carried out in detail in the first appendix. In the case $\kk=\FF_p$, we use the short exact sequences
\[1\to F_{n-1}\to P_n\to P_{n-1}\to 1\]
where the map $P_n \to P_{n-1}$ removes the last strand of a pure braid. The argument is then an induction on $n$ using a result of Nakamura (see \cite[Proposition 1.2.4 and 1.2.5]{nakamura}) that shows that, under reasonable assumptions, good groups are stable under extensions. The case $P_2\cong\ZZ$ is the base case of our induction and follows from the previous proposition.
\end{proof}

\section{Construction of certain homotopy left adjoints}

Our plan in the coming sections is to study the completion of configuration categories and how it interacts with taking tensor products. To achieve that we rely on the existence of certain localization functors (i.e. certain left adjoints). The purpose of this rather technical section is to establish such results. This is a delicate issue since the categories $\ProS_C$ are not cofibrantly generated, and so we cannot directly appeal to the usual Bousfield localization theorems.

\medskip
Throughout this section, we take $\cat{M}$ to be a cocombinatorial model category. The categories $\ProS_{\kk}$ are not quite cocombinatorial, for size reasons. Nevertheless, in Proposition \ref{prop:localization} at the end of this section, we will explain how the argument can be extended to these model categories.

\medskip
Let $U$ be a category with finite Hom sets. The category $\fun(U^{op},\cat{M})$ has a projective and injective model structure. Indeed, the category $\cat{M}$ being cocombinatorial means that the projective/injective model structure on $\fun(U^{op},\cat{M})$ is just the opposite of the injective/projective model structure on $\fun(U,\cat{M}^{op})$ which exists since $\cat{M}^{op}$ is combinatorial.

Now, consider a set $S$ of morphisms of $\fun(U^{op},\fin)$. Denote by $\tilde{U}$ the category $\fun(U^{op},\fin)$. The Yoneda embedding $y:U\to \tilde{U}$ induces a functor
\[y^* : \fun(\tilde{U}^{op},\cat{M})\to \fun(U^{op},\cat{M})\]
by precomposition. There is also a projective and an injective model structure on $\fun(\tilde{U}^{op},\cat{M})$ and $y^*$ preserves injective cofibrations and injective trivial cofibrations. It follows that the right Kan extension functor 
\[y_*:\fun(U^{op},\cat{M})\to \fun(\tilde{U}^{op},\cat{M})\]
has a right derived functor $\RR y_*$. This has the following practical description. Let $v$ be an object in $\fun(U^{op},\mathsf{Fin})$, and take a resolution $V_{\bullet}$ of $v$ by representables. Then, given $X \in \fun(U^{op},\cat{M})$, the value of $\RR y_*(X)$ at $v$ can be calculated as the homotopy limit of the simplicial object $X(V_\bullet)$.

\begin{defn}\label{defn:Slocal}
We say that an object $X$ of $\fun(U^{op},\cat{M})$ is \emph{$S$-local} if $\RR y_*(X)(f)$ is a weak equivalence in $\cat{M}$ for all morphisms $f$ in $S$.
\end{defn}

\begin{prop}
There exists a left model structure on $\fun(U,\cat{M})$ in which the cofibrations are the projective cofibrations, the fibrant objects are the fibrant objects $X$ in $\fun(U,\cat{M})$ that are $S$-local and the weak equivalences are the maps $f:V\to W$ such that for any $S$-local object $Z$, the induced map
\[[W,Z]\to [V,Z]\]
is a bijection, where $[-,-]$ denotes homotopy classes of maps in $\fun(U^{op},\cat{M})$.
\end{prop}

\begin{proof}
We argue as in the proof of \cite[Proposition 2.9]{Horel}. Let $\mathcal{K}$ be the full subcategory spanned by the $S$-local objects. According to \cite[Theorem 5.22]{barwicklocalization} it is enough to prove that $\mathcal{K}$ is stable under homotopy limits and weak equivalences and that it is coaccessible and coacessibly embedded in $\fun(U,\cat{M})$. By construction $\mathcal{K}$ is stable under homotopy limits and weak equivalences. By \cite[Proposition 2.5]{barwicklocalization}, there exists a regular cardinal $\kappa$ such that
\begin{itemize}
\item The $\kappa$-cofiltered limits in $\fun(U^{op},\cat{M})$ are homotopy limits.
\item There exists a $\kappa$-coacessible cofibrant replacement in $\fun(U^{op},\cat{M})_{inj}$
\item The weak equivalences in $\cat{M}^S$ form a $\kappa$-coaccessibly embedded $\kappa$-coaccessible subcategory of $\fun([1],\cat{M}^S)$ (where $[1]$ denotes the poset $\{0<1\}$).
\end{itemize} 
It follows first that $\mathcal{K}$ has $\kappa$-cofiltered limits and that the inclusion in $\fun(U^{op},\cat{M})$ preserves those. It remains to show that $\mathcal{K}$ is coaccessible. To do this we consider the composition
\[A:\fun(U^{op},\cat{M})\xrightarrow{\RR y_*}\fun(\tilde{U}^{op},\cat{M})\xrightarrow{ev_S}\fun([1],\cat{M})^S\]
where $ev_S$ evaluates a functor $X:\tilde{U}^{op}\to \cat{M}$ on the maps of $S$. The functor $A$ preserves $\kappa$-cofiltered limits. Moreover, $\mathcal{K}$ is by definition the full subcategory of $\fun(\tilde{U}^{op},\ProS)$ spanned by the objects $X$ such that $A(X)$  belongs to $(w\cat{M})^S\subset \fun([1],\cat{M})^S$. Then \cite[Corollary A.2.6.5]{lurietopos} yields the result.
\end{proof}

We denote this left model structure by $\fun_S(U^{op},\cat{M})$. By definition of a left Bousfield localization, the identity functor is a left Quillen functor
\[\fun(U^{op},\cat{M})\to\fun_S(U^{op},\cat{M})\]
It thus follows from \cite[Proposition 1.13]{barwicklocalization} that we have an adjunction at the level of homotopy categories. We record this in the following proposition.

\begin{prop}
The inclusion of the category of $S$-local objects in the category $\fun(U, \cat{M})$ has a homotopy left adjoint.
\end{prop}

We now extend this proposition to $\ProS_\kk$. For a set of maps $S$ in the category $\fun(U, \ProS_\kk)$, we define $S$-local objects as in Definition \ref{defn:Slocal}.

\begin{prop}\label{prop:localization}
The inclusion of the category of $S$-local objects in the category $\fun(U, \ProS_\kk)$ has a homotopy left adjoint.
\end{prop}

\begin{proof}
We know from Theorem \ref{theo : interpretation of proSC} that the $\infty$-category underlying $\ProS_\kk$ is equivalent to the $\infty$-category of pro-objects on $\S^{cofin}_{\kk}$. Therefore $\ProS_\kk$ is copresentable (i.e. its opposite is presentable). By a theorem of Dugger and Lurie (see \cite[Proposition A.3.7.6]{lurietopos}) on the classification of presentable $\infty$-categories, it follows that there exists a zig-zag of Quillen equivalences connecting $\ProS_\kk$ to $\cat{M}$ with $\cat{M}$ cocombinatorial. Using this zig-zag and the previous proposition, it is straightforward to construct the desired left adjoint.
\end{proof}

\begin{rem}
There is an explicit cocombinatorial model category that is equivalent to $\ProS_{\FF_p}$ due to Morel (see \cite{morelensembles} and \cite[Section 7]{bhh} for a comparison result).
\end{rem}

\section{A Boardman-Vogt product and its completions}

In \cite{BoavidaWeissBV}, a Boardman-Vogt type tensor product $\boxtimes$ was constructed relating the configuration categories of two manifolds to the configuration category of their product. In this section, we transfer this construction to the setting of pro-spaces.

\begin{defn}
A surjective map $f:\uli m \to \uli n$ is called \emph{selfic} if the function $\uli n \to \uli m$ defined by $x \mapsto \min f^{-1}(x)$ is increasing.
\end{defn}

\begin{defn}
Let $\boxfin$ denote the category whose objects are selfic maps ${\uli r \gets \uli k \to \uli s}$ such that the induced map $\uli k \to \uli r \times \uli s$ is injective. There are three projection maps $(p_0, p_1, p_2) : \boxfin \to \fin$ which map an object $\uli r \gets \uli k \to \uli s$ to $\uli r$, $\uli k$ and $\uli s$, respectively.
\end{defn}

In what follows, we let $\mathbf{C}$ be a model category. Given two contravariant functors $X, Y$ from $\simp(\fin)$ to $\mathbf{C}$, pullback along $(p_0, p_2)$ defines a contravariant functor on $\simp(\boxfin)$, which we can then left Kan extend along $p_1$ to obtain a new contravariant functor on $\simp(\fin)$. We denote this resulting functor by $X \boxtimes^\pre Y$.

\begin{defn} A functor $X : \simp(\fin)^\op \to \mathbf{C}$ is \emph{conservative} if for every map $\sigma \to \tau$ in $\simp(\fin)$ corresponding to a surjection $[m] \to [n]$ in $\Delta$, the induced map $X(\tau) \to X(\sigma)$ is a weak equivalence.
\end{defn}

Let $\Lambda$ be a conservatization functor. That is, $\Lambda$ is an endofunctor on the category $\fun(\simp(\fin)^\op, \mathbf{C})$ with the property that, for any $X$, $\Lambda X$ is conservative and is universal with respect to that property (i.e. it is a homotopy left adjoint). Such a functor exists for $\mathbf{C} = \ProS_\kk$ by Proposition \ref{prop:localization}. Such a functor also exists for $\mathbf{C} = \S$ (or more generally for any combinatorial model category) by standard Bousfield localizations techniques.

\begin{defn}
We define the bifunctor $\boxtimes$ on $\mathrm{Fun}(\simp(\fin)^\op,\mathbf{C})$ by the formula
\[X \boxtimes Y = \Lambda (X \boxtimes^{pre} Y) \; .\]
\end{defn}

The main goal of this section is to prove the following completed variant of \cite[Thm. 2.8]{BoavidaWeissBV}.

\begin{thm}\label{thm:con-pro-box} Suppose $M$ and $N$ are two connected topological manifolds. Then there is a natural weak equivalence 
\[
\config(M)^{\wedge} \boxtimes \config(N)^{\wedge} \to \config(M \times N)^{\wedge} \; 
\]
of functors $\simp(\fin)^{\op} \to \Pro(\S)_{\kk}$.
\end{thm}

This relies on the following proposition.

\begin{prop}\label{prop:boxcommute}
Let $X$ and $Y$ be two functors $\simp(\fin)^\op \to \S$. Then the natural map
\[
{(X \boxtimes Y)}^\wedge \xrightarrow{} {X}^\wedge \boxtimes {Y}^\wedge  \; .
\]
is an objectwise weak equivalence of functors $\simp(\fin)^\op \to \ProS_\kk$ .
\end{prop}

\begin{proof}
By Proposition \ref{prop : completion commutes with product}, we have that
\[
(X(\alpha) \times Y(\alpha))^\wedge \xrightarrow{\sim} X(\alpha)^\wedge \times {Y(\alpha)}^\wedge
\]
 for each $\alpha \in \simp(\fin)$. It follows that completion of $(p_0, p_2)^*(X \times Y)$ agrees with $(p_0,p_2)^*(X^{\wedge},Y^{\wedge})$. Moreover, it is easy to see that completion commutes with the extension to $\simp(\fin)$ along $p_1 : \simp(\boxfin) \to \simp(\fin)$. Indeed, the diagram of left adjoints
\[
	\begin{tikzpicture}[descr/.style={fill=white}]
	\matrix(m)[matrix of math nodes, row sep=2.5em, column sep=2.5em,
	text height=1.5ex, text depth=0.25ex]
	{
	\fun(\simp(\boxfin)^\op, \S) & \fun(\simp(\boxfin)^\op, \ProS_\kk) \\
	\fun(\simp(\fin)^\op, \S) & \fun(\simp(\fin)^\op, \ProS_\kk) \\
	};
	\path[->,font=\scriptsize]
		(m-1-1) edge node [auto] {$\wedge$} (m-1-2)
		(m-2-1) edge node [auto] {$\wedge$} (m-2-2)
		(m-1-1) edge node [left] {$(p_1)_!$} (m-2-1)
		(m-1-2) edge node [auto] {$(p_1)_!$} (m-2-2);
	\end{tikzpicture}
\]
commutes, because the corresponding diagram of right adjoints commutes trivially. To sum up, completion commutes with $\boxtimes^{\pre}$. It remains to see that completion commutes with $\Lambda$. By  construction, $\Lambda$ is the left adjoint to the inclusion of conservative objects. By the same reasoning with adjunctions as before, it suffices to see that the homotopy right adjoint to completion, $\RR\Mat$, sends conservative objects to conservative objects. But this is clear.
\end{proof}

\begin{proof}[Proof of Theorem \ref{thm:con-pro-box}]
We apply Proposition \ref{prop:boxcommute} and obtain a weak equivalence
\[
(\config(M)\boxtimes\config(N))^\wedge \to \config(M)^\wedge \boxtimes \config(N)^\wedge  \;.
\]
The proof then follows from the main result of \cite{BoavidaWeissBV} which asserts the existence of a map
\[\config(M) \boxtimes \config(N) \to \config(M \times N) \]
which is a weak equivalence.
\end{proof}

\section{The Grothendieck-Teichm\"{u}ller group}\label{sec:gt-on-En}

In this section, we recall the definition of the Grothendieck-Teichm\"{u}ller group and its relation to the little $2$-disks operad. There are two flavors of Grothendieck-Teichm\"{u}ller groups that correspond to the two completions that we are considering. 

\medskip
We begin by recalling the construction of the operad of \emph{parenthesized braids}, denoted $\PaB$. This is an operad in the category of groupoids. The set of objects of $\PaB(n)$ is the set of planar binary trees with $n$ leaves labelled from $1$ to $n$. A morphism between labelled binary trees $p \to q$ is a braid connecting the set of leaves of $p$ to the set of leaves of $q$ in a way that respects the labellings. The operad structure is given by cabling, i.e. replacing strands by braids.

The completion (in the sense of Section \ref{sec:goodness}) of a product of groupoids with finitely many objects is the product of the completions. Therefore, by applying completion arity-wise to $\PaB$ we obtain an operad in $\Pro(\mathbf{Gpd})_{\kk}$, denoted $\PaB^{\wedge}_{\kk}$.

\begin{defn}\label{defn : GT}
The Grothendieck-Teichm\"{u}ller group $\GT_{\kk}$ is the group of automorphisms of $\PaB^{\wedge}_{\kk}$ that fix the objects. 
\end{defn}

This definition is essentially due to Drinfeld \cite{drinfeldquasi} (see also \cite{fressebook} for a thorough discussion).

\medskip
By taking classifying spaces, we get an action of $\GT_\kk$ on the operad $B(\PaB^{\wedge}_\kk)$. To explain the relation with the little disks operad, we need a notion of $\infty$-operads in pro-spaces. We use the following.

\begin{defn}\label{def : infinity operads in pro-spaces}
An $\infty$-operad in $\ProS_\kk$ is a functor $X : \mathsf{Tree}^\op \to \ProS_\kk$ such that $X_\eta$ is contractible and for every tree $T$ the map
\[
X_T \to \prod^{h}_{v \in T} X_{C_{|v|}}
\]
is a weak equivalence, where the homotopy product runs over the vertices of $T$.
\end{defn}
The object $(B\PaB)^\wedge_\kk$ obtained by applying completion arity-wise to $B\PaB$ is an $\infty$-operad in $\ProS_\kk$ (not a strict operad) by Proposition \ref{prop : completion commutes with product}. It is a model for $(E_2)^\wedge_\kk$, the completion of $E_2$. There is a canonical map
\[
(B\PaB)^\wedge_\kk \to B(\PaB^{\wedge}_\kk)
\]
which is a weak equivalence of $\infty$-operads in pro-spaces by Proposition \ref{prop : braid groups are good}. Therefore, we have a map
\[
\GT_\kk \to \aut^h((E_2)^\wedge_\kk) \; .
\]
This is in fact an isomorphism on $\pi_0$ by \cite{fressebook} and \cite{Horel}, though we will not need that.

\medskip
We end this section with a completed analog of Corollary \ref{cor:aut}. Recall the inclusion $j :  \simp(\fin) \to \mathsf{Tree}$ from section \ref{sec:confcats}.
\begin{prop}\label{prop:complete-op-conf}
The restriction map
\[
j^* : \aut^h((E_n)_{\kk}^\wedge) \to \aut^h(\config(\RR^n)^\wedge_{\kk})
\]
is a weak equivalence for all $n$.
\end{prop}
\begin{proof}
As in section \ref{sec:confcats}, we are considering an adjunction
\[
j_! : \fun(\simp(\fin)^\op, \ProS_{\kk}) \leftrightarrows \cat{R} : j^* \; .
\]
Here $\cat{R}$ denotes the category $\fun(\mathsf{Tree}^\op, \ProS_{\kk})$ equipped with the localized model structure whose fibrant objects are the $\infty$-operads in pro-spaces $X$ having $X_{C_0}$ and $X_{C_1}$ weakly pro-contractible. This category is a localization of the model category $\fun(\mathsf{Tree}^\op, \ProS_{\kk})$, which exists by Proposition \ref{prop:localization}. The adjunction $(j_!, j^*)$ is a Quillen adjunction, and the statement of the proposition is equivalent to the following: the derived counit map
\[
j_! j^* ((E_n)_{\kk}^\wedge) \to (E_n)_{\kk}^\wedge
\]
is a weak equivalence in $\cat{R}$. Clearly, $j^* ((E_n)_{\kk}^\wedge) = j^* (E_n)_{\kk}^\wedge$. Also, the left adjoints $j_!$ and $(-)^\wedge_\kk$ commute, since their right adjoints $j^*$ and $\Mat$ trivially commute. Therefore, the derived counit map above is the completion of the derived counit map $j_! j^* E_n \to E_n$ which we know to be a weak equivalence by Theorem \ref{thm: from simpfin to omega}.
\end{proof}

\section{The box product and the join}

The theorem below (Theorem \ref{thm:main-spheres}) will be important to us as a way to generate interesting automorphisms of the completion of the little disks operads. In order to simplify notation, in this section $X^\wedge$ denotes the $\kk$-completion of a space $\kk=\FF_p$ or $\QQ$, in the sense of Definition \ref{defn:completion}.

The value of $\config(\RR^n)^{\wedge}$ at the finite set $\uli{2}$ is weakly equivalent to the completed sphere $(S^{n-1})^{\wedge}$. In particular, any homotopy automorphism of $\config(\RR^n)^{\wedge}$ induces a homotopy automorphism of $(S^{n-1})^{\wedge}$.

\begin{thm}\label{thm:main-spheres}
The restriction map 
\[
\aut^h(\config(\RR^n)^\wedge) \to \aut^h({(S^{n-1})}^\wedge)
\]
is surjective on $\pi_0$.
\end{thm}

The proof of this theorem relies on the following proposition.

\begin{prop}\label{prop:bi-spheres}
The following diagram
\[
	\begin{tikzpicture}[descr/.style={fill=white}, baseline=(current bounding box.base)]
	\matrix(m)[matrix of math nodes, row sep=2.5em, column sep=2.5em,
	text height=1.5ex, text depth=0.25ex]
	{
	\aut^h(\config(\RR^n)) \times \aut^h(\config(\RR^d)) & \aut^h(S^{n-1}) \times \aut^h(S^{d-1}) \\
	\aut^h(\config(\RR^{n+d})) & \aut^h(S^{n+d-1}) \\
	};
	\path[->,font=\scriptsize]
		(m-1-1) edge node [auto] {} (m-1-2)
		(m-2-1) edge node [auto] {} (m-2-2)
		(m-1-1) edge node [left] {$\boxtimes$} (m-2-1)
		(m-1-2) edge node [auto] {\textup{join}} (m-2-2);
	\end{tikzpicture}
\]
where the rows are the restrictions to the finite set $\uli 2$, commutes up to homotopy. An analogous homotopy commutative square exists if we replace the configuration categories and the spheres by their completions.
\end{prop}

\begin{proof}[Proof of Theorem \ref{thm:main-spheres} assuming Proposition \ref{prop:bi-spheres}]
The case $n = 1$ is trivial. For $n = 2$ and in the $\FF_p$-case, the map
\[
\aut^h(\config(\RR^2)^\wedge) \to \aut^h({(S^{1})}^\wedge)
\]
is surjective on $\pi_0$ because the compostion
\[\mathrm{Gal}(\overline{\QQ}/\QQ)\to\GT_{\FF_p}\to \pi_0 \aut^h(\config(\RR^2)^\wedge)\to \pi_0\aut^h({(S^{1})}^\wedge)=\ZZ_p^{\times}\]
is equal to the cyclotomic character, which is surjective. In the rational case, a similar argument holds with $\GT_\QQ$. In that case, the surjectivity of the cyclotomic character is proved in \cite[Proposition 5.3]{drinfeldquasi}.

For $n \geq 2$, we view $\RR^n$ as $\RR^{2} \times \RR^{n-2}$ and consider the commutative square
\[
	\begin{tikzpicture}[descr/.style={fill=white}]
	\matrix(m)[matrix of math nodes, row sep=2.5em, column sep=2.5em,
	text height=1.5ex, text depth=0.25ex]
	{
	\aut^h(\config(\RR^2)^\wedge) & \aut^h({(S^{1})}^\wedge)  \\ 
	\aut^h(\config(\RR^{n})^\wedge) & \aut^h((S^{n-1})^\wedge) \\
	};
	\path[->,font=\scriptsize]
		(m-1-1) edge node [auto] {} (m-1-2)
		(m-2-1) edge node [auto] {} (m-2-2)
		(m-1-1) edge node [left] {$- \boxtimes \textup{id}$} (m-2-1)
		(m-1-2) edge node [auto] {$- \star \textup{id}$} (m-2-2);
	\end{tikzpicture}
\]
from Proposition \ref{prop:bi-spheres}. The right-hand map is bijective on $\pi_0$ and so the result follows. 
\end{proof}

We have in fact just proved Theorem E and E$^\prime$ from the introduction.
\begin{proof}[Proof of Theorem E and E$^\prime$]
The $\GT_{\kk}$-action on $\config(\RR^2)^{\wedge}_{\kk}$ induces a $\GT_{\kk}$-action on $\config(\RR^n)^\wedge_{\kk}$ by taking the box product with $\config(\RR^{n-2})^{\wedge}_{\kk}$ equipped with the trivial action. Then the commutative square of Proposition \ref{prop:bi-spheres} implies that the restriction of this action to $(S^{n-1})^{\wedge}_{\kk}$ is indeed given by the cyclotomic character.
\end{proof}

\begin{rem}
The non-completed version of Theorem \ref{thm:main-spheres} also holds. This can be seen from the fact that the composition $O(n) \to \aut^h(\config(\RR^n)^\wedge) \to \aut^h(S^{n-1})$ is the identity on $\pi_0$, where $O(n)$ denotes the $n^{th}$ orthogonal group.
\end{rem}

\begin{cor}\label{cor:spheres}
Let $\mu$ be a homotopy automorphism of $(S^{d-1})^\wedge$. Then there exists a homotopy automorphism $\mu^\sharp$ of $\config(\RR^d)^\wedge$ making the diagram
\[
	\begin{tikzpicture}[descr/.style={fill=white}, baseline=(current bounding box.base)]
	\matrix(m)[matrix of math nodes, row sep=2.5em, column sep=2.5em,
	text height=1.5ex, text depth=0.25ex]
	{
	\aut^h(\config(\RR^n)^\wedge) & \aut^h({(S^{n-1})}^\wedge)  \\ 
	\aut^h(\config(\RR^{n+d})^\wedge) & \aut^h((S^{n+d-1})^\wedge) \\
	};
	\path[->,font=\scriptsize]
		(m-1-1) edge node [auto] {} (m-1-2)
		(m-2-1) edge node [auto] {} (m-2-2)
		(m-1-1) edge node [left] {$- \boxtimes \mu^\sharp$} (m-2-1)
		(m-1-2) edge node [auto] {$- \star \mu$} (m-2-2);
	\end{tikzpicture}
\]
commute up to homotopy.
\end{cor}
\begin{proof}
The existence of the lift $\mu^\sharp$ of $\mu$ follows from Theorem \ref{thm:main-spheres}. The commutativity of the diagram is a consequence of Proposition \ref{prop:bi-spheres}.
\end{proof}

In the rest of the section, we prove Proposition \ref{prop:bi-spheres}. In fact, we will prove a more general statement which holds for any two manifolds $M$ and $N$. In order to state it, we need a definition.

\medskip
Let $A$ be the subcategory of $\simp(\fin)$ consisting of two objects $\uli 1$ and $\uli 2$ and two non-identity morphisms $\uli 1 \to \uli 2$ selecting the two elements of $\uli 2$. The restriction of $\config(M)$ to $A$ is simply a diagram of spaces
\[
\emb(\uli 2, M) \to M \times M
\]
where the map forgets each of the points.

Now, consider the square of inclusion maps (over the space $M \times M \times N \times N$):
 \begin{equation}\label{eq:square2pts}
	\begin{tikzpicture}[descr/.style={fill=white}, baseline=(current bounding box.base)]
	\matrix(m)[matrix of math nodes, row sep=2.5em, column sep=2.5em,
	text height=1.5ex, text depth=0.25ex]
	{
	\emb(\uli 2, M) \times \emb(\uli 2, N) & M \times M \times \emb(\uli 2, N)  \\
	\emb(\uli 2, M) \times N \times N & \emb(\uli 2, M \times N) \; .\\
	};
	\path[->,font=\scriptsize]
		(m-1-1) edge node [auto] {} (m-2-1)
		(m-2-1) edge node [auto] {} (m-2-2)
		(m-1-1) edge node [auto] {} (m-1-2)
		(m-1-2) edge node [auto] {} (m-2-2);
	\end{tikzpicture}
\end{equation}
This square is clearly a pushout and, in fact, also a homotopy pushout. Moreover, as we will explain below, the map from the homotopy pushout to the product $M \times M \times N \times N$ models the map
\[
\config(M) \boxtimes \config(N)(\uli 2) \to \config(M \times N)(\uli 1) \times \config(M \times N)(\uli 1)
\]
which is induced by the two inclusions $\uli 1 \to \uli 2$. Assuming this for the moment, we can define a map
\[
J : \End^h(\config(M)|_{A}) \times \End^h(\config(N)|_{A}) \to \End^h(\config(M) \boxtimes \config(N)|_{A}) \;
\]
that sends a pair $(f, g)$ consisting of maps $f(\uli 2) \to f(\uli 1) \times f(\uli 1)$ and $g(\uli 2) \to g(\uli 1) \times g(\uli 1)$ to the map
\[
f(\uli 2) \times g(\uli 1)^{\times 2} \; \coprod^h_{f(\uli 2) \times g(\uli 2)} \; f(\uli 1)^{\times 2} \times g(\uli 2) \longrightarrow f(\uli 1)^{\times 2} \times g(\uli 1)^{\times 2} \; .
\]
We can now state the promised generalization of Proposition \ref{prop:bi-spheres}.

\begin{prop}\label{prop:bi-man}
The following diagram
\[
	\begin{tikzpicture}[descr/.style={fill=white}, baseline=(current bounding box.base)]
	\matrix(m)[matrix of math nodes, row sep=2.5em, column sep=2.5em,
	text height=1.5ex, text depth=0.25ex]
	{
	\End^h(\config(M)) \times \End^h(\config(N)) & \End^h(\config(M)|_A) \times \End^h(\config(N)|_A) \\
	\End^h(\config(M) \boxtimes \config(N)) & \End^h(\config(M) \boxtimes \config(N)|_A) \\
	};
	\path[->,font=\scriptsize]
		(m-1-1) edge node [auto] {\textup{restr.}} (m-1-2)
		(m-2-1) edge node [auto] {\textup{restr.}} (m-2-2)
		(m-1-1) edge node [left] {$\boxtimes$} (m-2-1)
		(m-1-2) edge node [auto] {$J$} (m-2-2);
	\end{tikzpicture}
\]
commutes.
\end{prop}

So far, we have not used any particular construction of the conservatization functor $\Lambda$. To prove this proposition, we will need to use a formula for $\Lambda$ from \cite[Section 8]{BoavidaWeissLong}, which we now describe. As before, let $X$ and $Y$ be contravariant functors on $\simp(\fin)$. Given an object in $\simp(\fin)$, i.e. a string $\alpha := (k_1 \gets \dots \gets k_n)$, the value of $X \boxtimes^\pre Y$ at $\alpha$ is the space of triples $(\beta, x, y)$ where $\beta$ is an object of $\simp(\boxfin)$ such that $p_1(\beta) = \alpha$, $x \in X(p_0(\beta))$ and $y \in X(p_2(\beta))$. Then $X \boxtimes Y := \Lambda(X\boxtimes^\pre Y)$ at a finite set $\uli k$ is the realization of the simplicial space whose space of $n$-simplices is the evaluation
\[
(X \boxtimes^\pre Y)(\uli k \xleftarrow{\id} \dots \xleftarrow{\id} \uli k)
\]
at the string consisting of $(n-1)$ identity maps.
\begin{proof}[Proof of Proposition \ref{prop:bi-man}]

The commutativity of the square is immediate once we unravel $\config(M)\boxtimes\config(N)(\uli 2)$ as the homotopy pushout of the square (\ref{eq:square2pts}).

We specialize the formula above to $k = 2$. Since there is only one selfic map $\uli 2 \to \uli 2$ (the identity), the simplicial space defining $(X \boxtimes Y)(\uli k)$ has, in degree zero, three summands corresponding to objects $\beta = (\uli r \gets \uli k \to \uli s)$ in $\boxfin$ where $(r,s)$ is of the form $(2,2)$, $(2,1)$ and $(1,2)$. These summands are respectively $X(\uli2) \times Y(\uli2)$, $X(\uli2) \times Y(\uli 1)$ and $X(\uli 1) \times Y(\uli2)$. In degree $1$, we have two maps
\[
(2,1) \gets (2,2) \quad , \quad (1,2) \gets (2,2)
\]
and the three identity maps. (In $\boxfin$ there are no non-trivial morphisms to $(2,2)$.)
More generally, in degree $n$, there will be a summand for each string
\[
(r_1, s_1) \gets \dots \gets (r_n, s_n)
\] subject to the condition: if $r_i = 1$ for some $i$ then $r_j = 1$ for all $j \leq i$ and $s_k = 2$ for all $k$ (and the analogous condition with $r$ and $s$ exchanged). In particular, all the simplices of dimension two or higher in the simplicial space defining $(X \boxtimes Y)(\uli k)$ are degenerate. Therefore, $(X \boxtimes Y)(\uli k)$ is identified with the homotopy coequalizer of the diagram
\begin{center}
$X(\uli 1 \gets \uli 2) \times Y(\uli2) \coprod X(\uli2) \times Y(\uli1 \gets \uli2)$ \\
$\downdownarrows$ \\
$X(\uli1) \times Y(\uli2) \coprod X(\uli2) \times Y(\uli2) \coprod X(\uli2) \times X(\uli1) \; .$
\end{center}

Now we set $X = \config(M)$ and $Y = \config(N)$. Then $X(\uli1) = M$, $X(\uli2) = \emb(\uli 2, M)$ and similarly for $Y$. To describe $X(\uli 2 \to \uli 1)$, take the normal bundle of the diagonal inclusion $M \hookrightarrow M \times M$ and let $S(M)$ be its unit sphere bundle. It follows from the description of the morphism spaces in $\config(M)$ from \cite{BoavidaWeissLong} that ${X(\uli 1 \gets \uli 2) \simeq S(M)}$. Under this identification, the map $X(\uli 1 \gets \uli 2) \to X(\uli 2)$ corresponds the inclusion $S(M) \hookrightarrow \emb(\uli 2, M)$ given by the tubular neighborhood theorem. Similarly for $Y$.

These considerations lead us to identify $(X \boxtimes Y)(\uli k)$ with the standard formula computing the iterated homotopy pushout of the diagram
\[
	\begin{tikzpicture}[descr/.style={fill=white}, baseline=(current bounding box.base)]
	\matrix(m)[matrix of math nodes, row sep=2.5em, column sep=2.5em,
	text height=1.5ex, text depth=0.25ex]
	{
	\emb(\uli 2, M) \times N &  & \\
	\emb(\uli 2, M) \times S(N) & \emb(\uli 2, M) \times \emb(\uli 2, N) & \\
	& S(M) \times \emb(\uli 2, N)  & M \times \emb(\uli 2, N) \\
	};
	\path[->,font=\scriptsize]
		(m-2-1) edge node [auto] {} (m-1-1)
		(m-2-1) edge node [auto] {} (m-2-2)
		(m-3-2) edge node [left] {} (m-2-2)
		(m-3-2) edge node [auto] {} (m-3-3);
	\end{tikzpicture}
\]

 To complete the proof, we observe that there is a homotopy pushout
\[
	\begin{tikzpicture}[descr/.style={fill=white}, baseline=(current bounding box.base)]
	\matrix(m)[matrix of math nodes, row sep=2.5em, column sep=2.5em,
	text height=1.5ex, text depth=0.25ex]
	{
	\emb(\uli 2, M) & M \times M \\
	S(M) & M  \\
	};
	\path[->,font=\scriptsize]
		(m-1-1) edge node [auto] {} (m-1-2)
		(m-2-1) edge node [auto] {} (m-1-1)
		(m-2-1) edge node [left] {} (m-2-2)
		(m-2-2) edge node [auto] {} (m-1-2);
	\end{tikzpicture}
\]
and similarly for $N$. It follows that $(X \boxtimes Y)(\uli 2)$ maps by a weak equivalence to the homotopy pushout of $(\ref{eq:square2pts})$.
\end{proof}

\section{The action in homology}

Let $X$ be an $\infty$-operad in pro-spaces. Then the collection
\[
H_*(X, \kk) := \{H_*(X(n),\kk)\}_{n\in\mathbb{N}}
\] 
forms a strict operad in pro-$\kk$-vector spaces. Here $X(n)$ denotes the value of $X$ at the $n^{th}$ corolla. This is defined as follows. Let $T$ be the tree consisting of a corolla with $n$ leaves with a corolla with $m$-leaves grafted onto the $i^{th}$ leaf. The $\circ_i$ product is given by
\[
H_*X(n) \otimes H_*X(m) \to H_*(X(n) \times X(m)) \gets H_*X(T) \to H_*X(n+m-1)\]
where the backwards arrow is an isomorphism by the Segal condition, and the first arrow is the K\"unneth isomorphism.

\begin{prop}\label{prop:homologycompl}
Let $\kk$ be $\QQ$ or $\FF_p$. The canonical map 
\[
H_*(E_{n}, \kk) \to H_*((E_{n})^\wedge_\kk,\kk)
\]
is an isomorphism of operads in the category of pro-objects in graded vector spaces.
\end{prop}

\begin{proof}
It suffices to check that this map is an isomorphism arity-wise. We claim, more generally, that for any connected space of finite $\kk$-type (i.e such that the $\kk$-homology groups are all finite dimensional), the canonical map
\begin{equation}\label{map}
H_i(X,\kk)\to H_i(X^{\wedge},\kk)
\end{equation}
is an isomorphism of pro-vector spaces. By definition, this is true after taking the topological dual (where for $\{V_i\}_{i\in I}$ a pro-vector space, its topological dual is the colimit of the diagram $i\mapsto \Hom_{\kk}(V_i,\kk)$). It is classical that the category of vector spaces is the ind-category of finite dimensional vector spaces. Consequently, the topological dual is an equivalence of categories between the pro-category of finite dimensional vector spaces and the opposite of the category of vector spaces. As such, it suffices to show that both the source and the target of the map (\ref{map}) are pro-finite dimensional vector spaces. This is true by assumption for the source.

Let $\{X_i\}_{i\in I}$ be a model for the completion $X^{\wedge}$. Since $X$ is connected, each of the maps $X\to X_i$ factors through one path component of $X_i$ that we denote by $\tilde{X}_i$. The map of pro-space $X\to\{\tilde{X}_i\}_{i\in I}$ is also a model for the completion of $X$. (Indeed, the target is also a pro-object in $\S_{\FF_p}^{cofin}$ or in $\S_{\QQ}^{cofin}$ and it is easy to check that the map also induces an isomorphism in cohomology.) We may thus assume without loss of generality that each $X_i$ in the completion of $X$ is connected. Accordingly, all we have to show is that a connected space $Y$ in $\S_{\kk}^{cofin}$ has finite $\kk$-type. This can be reduced to the case of $K(\kk,n)$ with $n>0$ using the fact that $Y$ is the limit of a finite tower of principal fibrations whose fibers are of the form $K(\kk,n)$ with $n>0$ (see the proof of Propositions \ref{prop : characterization of F_p cofin} and \ref{prop : characterization of Q cofin}). The fact that the spaces $K(\kk,n)$ have finite $\kk$-type is a classical computation.
\end{proof}

Let $\alpha$ be a unit in $\QQ$. Viewing $\alpha$ as an automorphism of $(S^{n-1})^\wedge_{\QQ}$, for $n \geq 2$, we may lift it to a homotopy automorphism of ${(E_n)}^\wedge_{\QQ}$ along the restriction map to arity two. Such a lift exists by Theorem \ref{thm:main-spheres} and Proposition \ref{prop:complete-op-conf}. Similarly, if $\alpha$ is a unit in $\FF_p$, we can first lift it to a unit in $\ZZ_p$ and then lift this to a homotopy automorphism $\alpha^\sharp$ of ${(E_n)}^\wedge_{\FF_p}$. 

\begin{prop}\label{prop : action on homology}
Let $\alpha$ be a unit in $\kk$. For $n \geq 2$, let $\alpha^\sharp$ be a homotopy automorphism of $E_n^\wedge$ which lifts $\alpha$. Then $\alpha^\sharp$ acts on $H_{k(n-1)}(E_n, \kk)$ as multiplication by $\alpha^k$.
\end{prop}

\begin{proof}
The homology of the little disks operad $H_*(E_n,\kk)$ is the Poisson operad, which is generated by operations in arity two. Therefore, it is enough to check that the induced action of $\alpha^\sharp$ in homology coincides with multiplication by $\alpha^k$ in $H_{k(n-1)}(E_n(2), \kk)$. But this holds by construction.
\end{proof}

\section{Purity implies formality}\label{section : purity implies formality}

Proposition \ref{prop : action on homology} constructs actions on the little disks operad which are \emph{pure}. In this section, we explain what this means and review some purity-implies-formality results from \cite{CiriciHorel}.

\medskip
By a chain complex we mean a homological non-negatively graded chain complex.

\begin{defn}
Let $N$ be a positive integer. A dg-operad $P$ (i.e an operad in chain complexes over a field) is called \emph{$N$-formal} if $P$ is connected to $H_*(P)$ via a zig-zag of maps of operads that induce isomorphisms in homology in degree $\leq N$.
\end{defn}

There is an alternative way of understanding $N$-formality using the truncation functor. This functor sends a chain complex $C_*$ to the chain complex $t_{\leq N}C_*$ whose value in degree $i$ is
\begin{itemize}
\item $C_i$ if $i<N$,
\item $C_N/d(C_{N+1})$ if $i=N$,
\item $0$ if $i>N$.
\end{itemize}
There is an obvious natural map $C_*\to t_{\leq N}C_*$, and the homology of $t_{\leq N}C_*$ coincides with the homology of $C_*$ in degrees $\leq N$ and is zero in degrees $>N$. Also, the functor $t_{\leq N}$ is lax symmetric monoidal. That is, there exists a natural transformation
\[t_{\leq N}C_*\otimes t_{\leq N}D_*\to t_{\leq N}(C\otimes D)\]
satisfying associativity and commutativity conditions. The following is straightforward (see \cite[Remark 5.4]{CiriciHorel}).

\begin{prop}
A dg-operad operad $P$ is $N$-formal if and only if the truncation $t_{\leq N}P$ is formal.
\end{prop}

\begin{defn}
Let $\kk$ be a field, let $u$ be a unit of $\kk$. Let $\alpha\neq 0$ be a rational number. We consider pairs $(C_*,\sigma)$ of a chain complex $C_*$ over $\kk$ with degreewise finite dimensional homology and $\sigma:C_*\to C_*$ an endomorphism. We call such a pair $(C_*,\sigma)$ \emph{$\alpha$-pure relative to $u$} (or just $\alpha$-pure if $u$ is implicit) if the following conditions are satisfied for all $n$:
\begin{enumerate}
\item If $\alpha n$ is not an integer, then $H_n(C_*)=0$.
\item If $\alpha n$ is an integer, then the unique eigenvalue of the endomorphism $H_n(\sigma)$ is $u^{\alpha n}$.
\end{enumerate}
\end{defn}

\begin{expl}
Here is a \emph{pure} situation. Let $k$ be the algebraic closure of a finite field of characteristic $p'$ with $p'\neq p$. Consider the chain complex $C_{*}=(C^{*}_{et}(\mathbb{P}^n_k,\FF_p))^ {\vee}$, that is, the linear dual of the \'etale cochains on the projective space of dimension $n$ on $k$. There is an action of the Frobenius automorphism $\varphi$ of $k$ on $C_*$. Then the pair $(C_*,\varphi)$ is $1/2$-pure relative to $p'$. This amounts to the following:
\begin{enumerate}
\item The homology of $C_*$ is concentrated in even degrees.
\item The action of $\varphi$ on $H_{2i}(C_*)$ is given by mutltiplication by $(p')^i$.
\end{enumerate}
\end{expl}

Now, we fix $u\in \kk^\times$ once and for all. Let $h\in\mathbb{N}\cup\infty$ be the order of $u$ in $\kk^\times$. We denote by $\cat{Ch}_*(\kk)^{\alpha}$ the category of pairs $(C_*,\sigma)$ that are $\alpha$-pure. This category is symmetric monoidal.

\begin{thm}\label{theo : purity implies formality}
Let $\alpha$ be a positive non-zero rational number with $\alpha<h$. Let $P$ be an operad in $\cat{Ch}_*(\kk)^{\alpha}$. Then $P$ is $N$-formal for $N= \lfloor \frac{h-1}{\alpha}\rfloor$. Moreover, one can choose the formality quasi-isomorphisms to be natural in $P$.
\end{thm}

\begin{proof}
We sketch a proof. The details are worked out in \cite{CiriciHorel}. 

Let $\cat{TMod}$ be the category of pairs $(V,\sigma)$ where $V$ is a finite dimensional vector space and $\sigma$ is an automorphism whose eigenvalues are powers of $u$. This is a symmetric monoidal abelian category. We denote by $\cat{Ch}_*(\cat{TMod})^{\alpha}$ the category of chain complexes in $\cat{TMod}$ that satisfy the following two conditions
\begin{enumerate}
\item If $\alpha n$ is not an integer, then $H_n(C_*)=0$.
\item If $\alpha n$ is an integer, then the unique eigenvalue of the endomorphism $H_n(\sigma)$ is $u^{\alpha n}$.
\end{enumerate}
Clearly, $\cat{Ch}_*(\cat{TMod})^{\alpha}$ is a subcategory of $\cat{Ch}_*(\kk)^{\alpha}$ (since a chain complex in finite dimensional vector spaces has degreewise finite dimensional homology).

Now, consider the category $gr^{(h)}\cat{Vect}$ of $\ZZ/(h)$-graded vector spaces (recall that $h$ is the order of $u$ in $\kk^\times$). By splitting a vector space as a direct sum of generalized eigenspaces, we get a symmetric monoidal functor $\cat{TMod}\to gr^{(h)}\cat{Vect}$. This induces a symmetric monoidal functor
\[\cat{Ch}_*(\cat{TMod})^{\alpha}\to\cat{Ch}_*(gr^{(h)}\cat{Vect})^{\alpha}\]
where the target category is the full subcategory spanned by chain complexes $C_*$ in $\ZZ/(h)$-graded vector spaces satisfying the following two conditions:
\begin{enumerate}
\item If $\alpha n$ is not an integer, then $H_n(C_*)=0$.
\item If $\alpha n$ is an integer, then $H_n(C_*)$ is concentrated in degree $\alpha n$ modulo $h$.
\end{enumerate}
We claim that the forgetful functor
\[U:\cat{Ch}_*(gr^{(h)}\cat{Vect})^{\alpha}\to\cat{Ch}_*(\kk)\]
is $N$-formal as a lax symmetric monoidal functor (for $N$ as in the statement of the theorem). This means that there is a zig-zag of natural lax monoidal transformations connecting $U$ to $H_*\circ U$. The proof of this fact is elementary linear algebra. It is done by producing an explicit such zig-zag (see \cite[Proposition 5.13]{CiriciHorel}). From this fact, we immediatly deduce that the forgetful functor
\[ \cat{Ch}_*(\cat{TMod})^{\alpha}\to\cat{Ch}_*(\kk)\]
is also $N$-formal as a lax symmetric monoidal functor since it is the composite of a lax symmetric monoidal functor with a formal lax symmetric monoidal functor. Therefore, any operad $P$ in $\cat{Ch}_*(\cat{TMod})^{\alpha}$ is sent to an $N$-formal dg-operad.

This is almost the complete proof of the theorem. The only issue is that our operad $P$ does not live in $\cat{Ch}_*(\cat{TMod})^{\alpha}$ but in the larger category $\cat{Ch}_*(\kk)^{\alpha}$. The healing observation is that the inclusion $\cat{Ch}_*(\cat{TMod})^{\alpha}\to \cat{Ch}_*(\kk)^{\alpha}$ is an equivalence of symmetric monoidal $\infty$-categories. This is the content of Theorem 4.7 in \cite{CiriciHorel}. This implies that the above argument works at the cost of having to replace dg-operads by $\infty$-operads (i.e. operads in the symmetric monoidal $\infty$-category of chain complexes). More precisely, this proves that $t_{\leq N}P$ and $t_{\leq N}H_*(P)$ are equivalent as $\infty$-operads in chain complexes. This equivalence can be rigidified to a zig-zag of quasi-isomorphisms between strict dg-operads using the results of Hinich \cite{Hinich}.
\end{proof}

We will also need a cohomological version of Theorem \ref{theo : purity implies formality}. For simplicity, we will only state it in characteristic zero, though it holds more generally. For $\kk$ a field of characteristic zero, we denote by $\cat{Ch}^*(\kk)$ the category of cohomologically graded chain complexes. For $u$ a unit in $\kk$ of infinite order and $\alpha$ a positive rational number different from zero, we denote by $\cat{Ch}^*(\kk)^{\alpha}$ the category whose objects are pairs $(C^*,\sigma)$ where $C^*$ is an object in $\cat{Ch}^*(\kk)$ with finite dimensional cohomology in each degree and $\sigma$ is an endomorphism of $C^*$ satisfying the following conditions.
\begin{enumerate}
\item If $\alpha n$ is not an integer, then $H^n(C^*)=0$.
\item If $\alpha n$ is an integer, then the unique eigenvalue of the endomorphism $H^n(\sigma)$ is $u^{\alpha n}$.
\end{enumerate}

\begin{thm}\label{thm : cohomological formality}
Let $u$ and $\alpha$ be as above. Then the forgetful functor
\[U:\cat{Ch}^*(\kk)^{\alpha}\to \cat{Ch}^*(\kk)\]
is formal as a lax monoidal functor. That is, there exists a zig-zag of quasi-isomorphisms of lax monoidal functors connecting $U$ and $H^*\circ U$.
\end{thm}

\begin{proof}
The proof is completely analogous to the proof above. We refer the reader to Proposition 6.1 of \cite{CiriciHorel} for more details.
\end{proof}

This theorem immediately implies that any algebra over a discrete operad in $\cat{Ch}^*(\kk)^\alpha$ is formal and moreover that the formality quasi-isomorphisms are natural.

\section{Absolute formality}

We now come to the proof of Theorem A.

\begin{thm}\label{thm:main}
The little $n$-disks operad is formal over $\QQ$ and $(n-1)(p-2)$-formal over $\FF_p$.
\end{thm}

Let $\kk$ be $\QQ$ or $\FF_p$ for $p$ a prime. Let $P$ be an operad in spaces. By K\"unneth, applying singular chains arity-wise produces a dg-operad $C_*(P,\kk)$. There are technicalities involved in making a similar statement when $P$ is a dendroidal object, since dendroidal objects in $\cat{Ch}_*$ do not model infinity-operads in chain complexes. This is not an essential issue, however, because of the proposition below.

\medskip
Let $\cat{Op}_{\infty} \ProS_\kk$ be the $\infty$-category of $\infty$-operads in $\ProS_\kk$ (as in Definition \ref{def : infinity operads in pro-spaces}) and let $\cat{OpCh}_*(\kk)$ be the $\infty$-category of dg-operads (that is, the localization of the category of dg-operads at the quasi-isomorphisms).

\begin{prop}\label{prop:chains-dend}
There exists a functor $C^\prime_* : \cat{Op}_{\infty}\ProS_\kk \to \cat{OpCh}_*$ together with a natural map
\[
C_*(P,\kk) \to C^\prime_*(P^\wedge,\kk)
\]
for $P$ an operad in spaces, which is a weak equivalence of dg-operads if $P(n)$ is of finite type for each $n$.
\end{prop}
We postpone the proof to the second appendix.

\begin{proof}[Proof of Theorem \ref{thm:main}]
For $n=1$ the little $n$-disks operad is obviously formal over both $\FF_p$ and $\QQ$ so we assume that $n\geq 2$. Let $u$ be a unit of $\kk$ of infinite order if $\kk=\QQ$, or a unit of order $p-1$ if $\kk=\FF_p$. By Proposition \ref{prop : action on homology} and the paragraph preceeding it, we can find an automorphism $u^{\sharp}$ of $C'_*(E_n^\wedge, \kk)$ making it $1/(n-1)$-pure, i.e. so that $C'_*(E_n^\wedge, \kk)$ is an operad in $\cat{Ch}(\kk)^{\alpha}$ for $\alpha = \frac{1}{n-1}$. By Theorem \ref{theo : purity implies formality}, it follows that $C'_*(E_n^\wedge, \kk)$ is formal in the desired range. We conclude that $C_*(E_n,\kk)$ is formal in the same range since, by Proposition \ref{prop:chains-dend}, we have a weak equivalence
$
C_*(E_n,\kk)\simeq C'_*(E_n^\wedge, \kk) \; .$
\end{proof}

\begin{rem}
The appeal to Proposition \ref{prop:chains-dend} is not necessary when $n\geq 3$. Indeed, in that case, by Proposition \ref{prop: completion of nilpotent spaces}, the dendroidal space $\RR\Mat(E_n^{\wedge})$ is a $\kk$-localization of $E_n$. Using Theorem \ref{thm: cm}, we can strictify this dendroidal space to an equivalent operad in spaces. Then, we can apply the ordinary chain functor to get a dg-operad that is equivalent to $C_*(E_n,\kk)$. Moreover, any automorphism of $E_n^\wedge$ will induce an automorphism of that dg-operad. The rest of the proof of Theorem \ref{thm:main} is unchanged.
\end{rem}

\begin{rem}
It seems that something like Proposition \ref{prop:chains-dend} is implicitly used in Petersen's paper \cite{Petersen}. Indeed, in the proof of the main Theorem of \cite{Petersen}, one needs to extract a dg-operad from an operad in pro-spaces, namely the classifying space of the completion of $\PaB$. The naive guess would be to first take the limit of the pro-spaces and then apply the standard singular chain functor but we believe that the resulting object would not be equivalent to the chains on the little $2$-disks operad. This is because the configuration spaces of points in $\mathbb{R}^2$ are not good in the sense of Bousfield-Kan.
\end{rem}

\section{Relative formality}

\begin{defn}
A map $u:P\to Q$ of dg-operads over a field $\kk$ is said to be \emph{formal} is it is connected to the map $H_*(u)$ by a zig-zag of weak equivalences in the arrow category of dg-operads.
\end{defn}

We may now prove Theorem B.

\begin{thm}
Let $n$ and $d$ be two positive integers with $d\geq 2$. Let $E_n\to E_{n+d}$ be a map between little disks operads. Then
\begin{enumerate}
\item[($\QQ$)] the map of dg-operads $C_*(E_n,\QQ)\to C_*(E_{n+d},\QQ)$ is formal,
\end{enumerate}
and
\begin{enumerate}
\item[(${\FF_p}$)] the map of dg-operads $C_*(E_n,\FF_p)\to C_*(E_{n+d},\FF_p)$ is $N$-formal, 
\end{enumerate}
where $N = \Delta(p-2)$ and $\Delta$ is the greatest common divisor of $n-1$ and $d$.
\end{thm}

\begin{proof}
A map of operads is formal if and only if a map homotopic to it is formal. The space of operad maps $E_n \to E_{n+d}$ is connected if $d \geq 2$ (see \cite[Proposition 10.8]{FTW} or \cite[Corollary 4.1.5]{goppl}). Therefore, we may assume that the map $E_n \to E_{n+d}$ under consideration is the one induced by the standard embedding $\RR^n\to\RR^{n+d}$.

We denote by $\kk$ the field $\FF_p$ or $\QQ$. Let $u$ and $\mu$ be two elements of $\kk^\times$. By Corollary \ref{cor:spheres} and Proposition \ref{prop : action on homology}, we can find an automorphism $u^\sharp$ of $E_n^{\wedge}$ and an automorphism $\mu^\sharp$ of $E_{d}^{\wedge}$ such that the following properties are satisfied:
\begin{itemize}
\item  $u^\sharp$ induces multiplication by $u$ on the vector space $H_{n-1}(E_n(2),\kk)$,
\item $u^\sharp\boxtimes\mu^\sharp$ induces multiplication by $u\mu$ on the vector space $H_{n+d-1}(E_{n+d}(2),\kk)$,
\item $u^{\sharp}$ and $u^{\sharp}\boxtimes\mu^{\sharp}$ define an endomorphism of the map $E_{n}^{\wedge}  \to E_{n+d}^{\wedge}$ in the homotopy category of the arrow category of $\Op_{\infty}\ProS_\kk$. This is represented by a diagram
\[
	\begin{tikzpicture}[descr/.style={fill=white}, baseline=(current bounding box.base)]
	\matrix(m)[matrix of math nodes, row sep=2.5em, column sep=2.5em,
	text height=1.5ex, text depth=0.25ex]
	{
	E_n^{\wedge} & E_{n+d}^{\wedge} \\
	E_n^{\wedge} & E_{n+d}^{\wedge} \; .\\
	};
	\path[->,font=\scriptsize]
		(m-1-1) edge node [left] {$u^\sharp$} (m-2-1)
		(m-2-1) edge node [auto] {} (m-2-2)
		(m-1-1) edge node [auto] {} (m-1-2)
		(m-1-2) edge node [auto] {$u^\sharp \boxtimes \mu^\sharp$} (m-2-2);
	\end{tikzpicture}
\]
\end{itemize}

The last bullet point amounts to the functoriality of $\boxtimes$, together with the observation that $E_n^{\wedge}\to E_{n+d}^{\wedge}$ comes from a map of configuration categories
\[
\textup{id} \boxtimes i : \config(\RR^n)^{\wedge}\boxtimes\config(\RR^0)^{\wedge}\to\config(\RR^n)^{\wedge}\boxtimes\config(\RR^d)^{\wedge}
\]
where $i$ is any map $\config(\RR^0)^{\wedge}\to \config(\RR^d)^{\wedge}$ (the space of such maps is contractible).

As in the absolute case, we want to be in a position to apply Theorem \ref{theo : purity implies formality}. Accordingly, the following task emerges: choose units $u$ and $\mu$, so that the commutative diagram of dg-operads
\[
	\begin{tikzpicture}[descr/.style={fill=white}, baseline=(current bounding box.base)]
	\matrix(m)[matrix of math nodes, row sep=2.5em, column sep=2.5em,
	text height=1.5ex, text depth=0.25ex]
	{
	C'_*(E_n^{\wedge},\kk) & C'_*(u^\sharp\boxtimes \mu^\sharp) \\
	C'_*(E_n^{\wedge},\kk) & C'_*(u^\sharp\boxtimes \mu^\sharp) \; .\\
	};
	\path[->,font=\scriptsize]
		(m-1-1) edge node [left] {$C'_*(u^\sharp)$} (m-2-1)
		(m-2-1) edge node [auto] {} (m-2-2)
		(m-1-1) edge node [auto] {} (m-1-2)
		(m-1-2) edge node [auto] {$C'_*(u^\sharp \boxtimes \mu^\sharp)$} (m-2-2);
	\end{tikzpicture}
\]
is a diagram in the category of $\alpha$-pure chain complexes relative to $v$, for some fixed choice of rational number $\alpha$ and unit $v$. That is, choose $u$ and $\mu$ so that the map 
\[
C_*(E_n, \kk) \to C_*(E_{n+d}, \kk)
\] 
together with the automorphism $(u^\sharp, u^\sharp \boxtimes \mu^\sharp)$ is $\alpha$-pure relative to $v$.

One solution is as follows. Let us treat the $\FF_p$ case. Pick a unit $v$ of $\FF_p$ of order $(p-1)$ and set
\[
u = v^{(n-1)/ \Delta} \quad \textup{ and } \quad \mu = v^{d / \Delta} \; .
\]
Then, in homological degree $j$, $v$ acts on both $H_j(E_n(2), \kk)$ and $H_{j}(E_{n+d}(2), \kk)$ as multiplication by $v^{j/\Delta}$. So we set $\alpha = 1/\Delta$, and conclude using Theorem \ref{theo : purity implies formality} that the map is $\Delta(p-2)$-formal, as required. The rational case is similar, but we pick a unit $v$ of infinite order.

This proof needs to be adapted slightly in the case $n=1$. In that case, we cannot realize any unit $u\in\kk^\times$. This is not a problem, we can pick $u^\sharp=\id, u=1$ and the argument above goes through without changing anything else. Note that in that case $\Delta=d$.
\end{proof}

\begin{rem}
Since formality of $C_*(E_n,\FF_p)\to C_*(E_{n+d},\FF_p)$ implies formality of the source, one should not expect a formality range that is better than $(n-1)(p-2)$. So the range of the theorem seems to be optimal when $(n-1)$ divides $d$. It is unclear to us whether the range can be sharpened if $(n-1)$ does not divide $d$.
\end{rem}

\section{Hopf formality}

In the rational case, we can prove a stronger formality result that captures both the formality of the dg-operad $C_*(E_n,\QQ)$ and the formality of each of the commutative (or more precisely $E_\infty$) differential graded algebras $C^*(E_n(k),\QQ)$ in a compatible way. This is Theorem C in the introduction.

\medskip
Let $\cat{CDGA}$ be the category of commutative differential graded commutative algebras over $\QQ$. Sullivan constructed a functor of piecewise polynomial differential forms
\[\Omega^*_{poly}:\S\to\cat{CDGA}^\op\]
that is naturally quasi-isomorphic to $C^*(-,\QQ)$ as an $E_\infty$-algebra. The value $\Omega^*_{poly}(X)$ captures all the rational information about the homotopy type of $X$ when $X$ is nilpotent of finite type.

\begin{defn}
An $\infty$-Hopf cooperad is a functor $X:\mathsf{Tree}^\op\to\cat{CDGA}^\op$ that satisfies the Segal condition and such that the value at $\eta$ is quasi-isomorphic to the terminal object of $\cat{CDGA}^\op$.
\end{defn}

Variants of this definition appeared in \cite[Section 3]{lambrechtsformality} and in \cite[Section 8.5]{ciricimixed}.

\begin{rem}
It is probably worth mentioning that the product in $\cat{CDGA}^\op$ is the coproduct in $\cat{CDGA}$ and is simply given by the tensor product of commutative algebras. The terminal object in $\cat{CDGA}^\op$ is the algebra $\QQ$ concentrated in degree $0$. Moreover, the cohomology of an $\infty$-Hopf cooperad is a strict Hopf cooperad by the K\"unneth isomorphism.
\end{rem}

For $X$ and $Y$ two spaces, we have a canonical map in $\cat{CDGA}$
\[\Omega^*_{poly}(X)\otimes\Omega^*_{poly}(Y)\to \Omega^*_{poly}(X\times Y)\]
which is a quasi-isomorphism by the K\"unneth isomorphism. From this observation, we easily deduce the following result.

\begin{prop}
Let $P$ be an $\infty$-operad. Then $\Omega^*_{poly}(P)$ is an $\infty$-Hopf cooperad.
\end{prop}

\begin{rem}
In his book \cite{fressebook}, Benoit Fresse introduces the notion of strict Hopf cooperads. These are operad objects in $\cat{CDGA}^{op}$. Fresse also constructs a functor from operads in spaces to Hopf cooperads. This functor is not obtained by simply applying $\Omega^*_{poly}$ aritywise, since $\Omega^*_{poly}$ is not a lax monoidal functor (it is in fact oplax) and therefore does not send operads to strict Hopf coooperads. It seems very plausible that the homotopy theory of Hopf cooperads studied by Fresse is equivalent to the homotopy theory of $\infty$-Hopf cooperads. We will not address this question here. Nevertheless, we point out that the $\infty$-Hopf cooperad $\Omega^*_{poly}(P)$ carries enough information to construct the rationalization of a topological operad $P$. 
Indeed, applying a spatial realization functor objectwise to $\Omega^*_{poly}(P)$ we obtain an $\infty$-operad in spaces that we can then turn into a strict operad by Theorem \ref{thm: cm}.
\end{rem}

The Hopf analogues of the theorems of the previous sections are the following.

\begin{thm}
For all $n$, the $\infty$-Hopf cooperad $\Omega_{poly}^*(E_n)$ is quasi-isomorphic to $H^*(E_n,\QQ)$ as an $\infty$-Hopf cooperad.
\end{thm}

\begin{proof}
We first extend $\Omega^*_{poly}$ to the category of pro-spaces via the following formula
\[\Omega^*_{poly}(\{X_i\}_{i\in I}):=\colimsub{i\in I}\Omega^*_{poly}(X_i)\]
We claim that this functor is naturally quasi-isomorphic to $C^*(-,\QQ)$. Indeed, since filtered colimits are exact, it is enough to prove this for a constant pro-space in which case this is a standard fact about $\Omega^*_{poly}(-)$. The canonical map of dendroidal objects in pro-spaces
\[E_n\to E_{n}^{\wedge}\]
induces a quasi-isomorphism of $\infty$-Hopf cooperads
\[\Omega^*_{poly}(E_n)\to \Omega^*_{poly}(E_{n}^{\wedge})\]
by Proposition \ref{prop:homologycompl}. It is therefore enough to prove the formality of the target. Pick an element $u$ in $\QQ^{\times}$ of infinite order and, using Proposition \ref{prop : action on homology}, lift it to an automorphism $u^\sharp$ of $E_n^{\wedge}$ such that the action of $u^\sharp$ on $H^k(\Omega^*_{poly}(E_{n}^{\wedge}))$ is multiplication by the number $u^{\ell}$ where $\ell = -k/(n-1)$. We can then apply Theorem \ref{thm : cohomological formality} and deduce that the functor $\Omega^*_{poly}(E_{n}^{\wedge}) : \mathsf{Tree}\to\cat{CDGA}$ is formal.
\end{proof}

\begin{thm}
Let $n\geq 1$ and $d\geq 2$ be two integers. Then the map 
\[
\Omega^*_{poly}(E_n)\to \Omega^*_{poly}(E_{n+d})
\]
induced by any map $E_n \to E_{n+d}$, is formal as a map of $\infty$-Hopf cooperads.
\end{thm}

\section{Formality of chains on configuration spaces}

We now prove Theorem D from the introduction.

\begin{thm}
Let $d\geq 3$, the dg-algebra $C^*(\mathrm{Conf}_n(\RR^d),\FF_p)$ is $(d-1)(p-2)$-formal.
\end{thm}

\begin{proof}
We pick a unit $u$ of $\FF_p$ of order $(p-1)$ and using Proposition \ref{prop : action on homology}, we lift it to an automorphism of $u^\sharp$ of $E_d^{\wedge}(n)\simeq \mathrm{Conf}_n(\RR^d)^{\wedge}$ which acts by multiplication by $u^k$ in cohomological degree $k(d-1)$. We can then apply \cite[Proposition 7.8]{CiriciHorel} and deduce that the dg-algebra $C^*(\mathrm{Conf}_n(\RR^d)^{\wedge},\FF_p)$ is $(p-2)(d-1)$ formal. This dg-algebra is quasi-isomorphic to $C^*(\mathrm{Conf}_n(\RR^d),\FF_p)$ by definition of completion.
\end{proof}

\begin{rem}
Contrary to the previous section, we do not know how to formulate a Hopf formality statement over $\FF_p$ that would include the above statement and the formality of $C_*(E_n,\FF_p)$ as a dg-operad. The problem is that the tensor product is no longer the coproduct in the category of dg-algebras over $\FF_p$. 
\end{rem}

\begin{rem}
In \cite[Theorem 8.13]{CiriciHorel} a similar result is obtained when $d$ is even. The theorem above applies in the odd dimensional case. The bound that we obtain here is sharper than the one in \emph{loc. cit.} in the even dimensional case.
\end{rem}

\begin{cor}
For $d\geq 3$, the dg-algebra $C^*(\mathrm{Conf}_n(\RR^d),\FF_p)$ is formal if $n\leq (p-1)$.
\end{cor}

\begin{proof}
The top non-vanishing cohomology group of $C^*(\mathrm{Conf}_n(\RR^d),\FF_p)$ is in degree $(d-1)(n-1)$. Therefore, using the previous theorem, this dg-algebra is formal if 
$(d-1)(n-1)\leq (d-1)(p-2)$.
\end{proof}

\appendix

\section{Goodness of the pure braid groups over $\QQ$}

In this section, we prove that the pure braid groups are good over $\QQ$ (in the sense of Definition \ref{defn:good-group}). We could not locate this precise statement in the literature, but we believe this result is known (\emph{c.f.}  \cite{papadimarational} or Proposition 14.2.4 in the second volume of \cite{fressebook} for similar, though not quite the same, statements).

\medskip
We denote by $\cat{nGrp}$ the category of nilpotent groups. Given a group $G$, we denote by $\Gamma^k G$ the $k$-th term of the lower central series filtration. The following terminology seems to be due to Quillen.

\begin{defn}
A nilpotent Malcev group is a group that is nilpotent and uniquely divisible. We denote the category of such groups by $\cat{nMal}$.
\end{defn}

The inclusion $\cat{nMal}\to\cat{nGrp}$ has a left adjoint. We follow Bousfield-Kan and denote this left adjoint by $G\mapsto G\otimes\QQ$. This is not an abuse of notation since, if $G$ is abelian, this left adjoint is precisely given by tensoring with $\QQ$. Quillen proves in \cite[Appendix A, Cor. 3.8]{quillenrational} that $G \otimes \QQ$ is identified with the Malcev completion of $G$. Note also that a nilpotent Malcev group whose abelianization is finite-dimensional over $\QQ$ is, by definition, the same as a $\QQ$-unipotent group (see Definition \ref{def : unipotent}).

\begin{prop}\label{prop:lcsadjoint}
For any group $G$, the functor 
\[
\Hom(G,-):\cat{nMal}\to\cat{Set}
\]
is pro-represented by the pro-object $\{(G/\Gamma^k G)\otimes\QQ\}_{k\in\mathbb{N}}$.

If $G$ has the property that $H_1(G,\QQ)$ is finite dimensional, then the pro-object 
\[
\{(G/\Gamma^k G)\otimes\QQ\}_{k\in\mathbb{N}}
\] 
is a model for the $\QQ$-completion $G^{\wedge}$.
\end{prop}

\begin{proof}
The first statement amounts to the isomorphism
\[
\Hom(\{G/\Gamma^kG \otimes \QQ\}_{k}, N) = \colimsub{k}\Hom(G/\Gamma^kG \otimes \QQ, N) \cong \Hom(G, N) \; .
\]
which holds, by adjunction, for every nilpotent Malcev group $N$.

In order to prove the second statement, it suffices to prove that for each $k$ the group $(G/\Gamma^k G)\otimes\QQ$ is $\QQ$-unipotent. It is clearly nilpotent and uniquely divisible, and its lower central series filtration has graded pieces of the form $\Gamma^\ell G / \Gamma^{\ell+1} G \otimes \QQ$, for $\ell < k$, and these are finite-dimensional by hypothesis.
\end{proof}

Proposition \ref{prop:lcsadjoint} says that the assignment $G \mapsto \{ G/\Gamma^k G \otimes\QQ\}_{k}$ is a model for the left adjoint to the inclusion $\Pro(\cat{nMal}) \to \Pro(\cat{Grp})$ and, under finiteness conditions on $H_1$, for the left adjoint to the inclusion $\Pro({\QQ\cat{-uni}}) \to  \Pro(\cat{Grp})$ where $\QQ\cat{-uni}$ denotes the category of $\QQ$-unipotent groups.

\begin{prop}
Good groups are stable under finite coproducts.
\end{prop}

\begin{proof}
We denote by $\star$ the coproduct of groups. Assume that $G$ and $H$ are good groups. We pick models $G^{\wedge}=\{G_i\}_{i\in I}$ and $H^{\wedge}=\{H_i\}_{i\in I}$ indexed by the same filtered category, which is always possible. In that case, one has that $\{G_i\star H_i\}_{i\in I}$ is a model for $(G\star H)^{\wedge}$. We know that $B(G\star H)\simeq BG\vee BH$. By goodness of $G$ and $H$, this implies that $(B(G\star H))^{\wedge}$ is weakly equivalent to the pro-object $\{BG_i\vee BH_i\}_{i\in I}$ which is itself weakly equivalent to $\{B(G_i\star H_i)\}_{i\in I}$.
\end{proof}

Now we can prove the main result.

\begin{prop}\label{prop : pure braid groups are Q-good}
The pure braid groups $P_n$ are good.
\end{prop}

\begin{proof}
We prove this by induction on $n$. First, we observe that $\mathbb{Z}$ is good. Indeed $\mathbb{Z}^{\wedge}\cong\QQ$ and the map $B\ZZ\to B\QQ$ is certainly a $\QQ$-cohomology equivalence. This means that $P_2$ is good which is the base case of our induction. Using the previous proposition, we can also conclude that all free groups $F_n$ are good. We assume that $P_n$ is good.

There is a split short exact sequence
\[1\to F_n\to P_{n+1}\to P_n\to 1\]
where the second map forgets the last strand of the braid. This short exact sequence gives us a Leray-Serre spectral sequence
\begin{equation}\label{eqn : LS1}
E_2^ {i,j}=H^i(P_n,H^j(F_n,\QQ))\implies H^{i+j}(P_{n+1},\QQ)
\end{equation}
It is well known that the action of $P_n$ on $F_n$ is trivial after abelianization. It follows, using Lemma \ref{lem:falkrandell} below, that the sequence
\[1\to F_n/\Gamma^kF_n\to P_{n+1}/\Gamma^kP_{n+1}\to P_n/\Gamma^kP_n\to 1 \]
is exact for each $k$. Applying $- \otimes \QQ$, the sequence
\[1\to F_n/\Gamma^kF_n\otimes\QQ\to P_{n+1}/\Gamma^kP_{n+1}\otimes\QQ\to P_n/\Gamma^kP_n\otimes\QQ\to 1 \; .\]
remains exact by \cite[V.2.4]{BK}. We then obtain a Leray-Serre spectral sequence associated to this short exact sequence, for each $k$. Since filtered colimits are exact, these spectral sequences assemble into a single spectral sequence

\begin{equation}\label{eqn : LS2}
E_2^ {i,j}=\colimsub{k}H^i(P_n/\Gamma^kP_n\otimes\QQ,H^j(F_n/\Gamma^kF_n\otimes\QQ,\QQ))
\end{equation}
\[
\implies \colimsub{k}H^{i+j}(P_{n+1}/\Gamma^kP_{n+1})\otimes\QQ,\QQ)\;.
\]
Now, there is a map from the spectral sequence (\ref{eqn : LS2}) to the spectral sequence (\ref{eqn : LS1}). We claim that this map induces an isomorphism on $E_2$-pages. Assuming this for the moment, it follows that the map $BP_{n+1}\to B(P_{n+1}^{\wedge})$ is an isomorphism on cohomology, and this finishes the proof.

We now give a proof of the claim that the two $E_2$-pages coincide. We denote the $E_2$-pages of the first and sequence spectral sequences by $E_2^{I}$ and $E_2^{II}$, respectively. There is an isomorphism
\[E_2^{II}=\colimsub{k\geq \ell}H^i(P_n/\Gamma^kP_n\otimes\QQ,H^j(F_n/\Gamma^\ell F_n\otimes\QQ,\QQ))\]
where the action of $P_n/\Gamma^k P_n$ on $F_n/\Gamma^\ell F_n$ is obtained by restricting the action of $P_n/\Gamma^\ell P_n$ on $F_n/\Gamma^\ell F_n$ along the quotient map $P_n/\Gamma^k P_n\to P_n/\Gamma^\ell P_n$. This isomorphism comes from the observation that the poset of pairs of integers $(k,\ell)$ with $k\geq \ell$ contains the subposet of pairs of the form $(k,k)$ and that this inclusion is terminal (alias cofinal). By the Fubini theorem for colimits, we can rewrite this isomorphism as
\[E_2^{II}=\colimsub{\ell}\colimsub{k\geq \ell}H^i(P_n/\Gamma^kP_n\otimes\QQ,H^j(F_n/\Gamma^\ell F_n\otimes\QQ,\QQ)) \; .\]
Using the goodness of the pure braid group $P_n$, we obtain an isomorphism
\[E_2^{II}=\colimsub{\ell}H^i(P_n,H^j(F_n/\Gamma^\ell F_n\otimes\QQ,\QQ)) \; ,\]
since the action of $P_n$ on $H^j(F_n/\Gamma^\ell F_n\otimes\QQ,\QQ)$ is $\QQ$-unipotent. (This follows by induction on $\ell$. The action of $P_n$ on $\Gamma^\ell F_n / \Gamma^{\ell+1} F_n \otimes \QQ$ is trivial, hence we can apply Lemma \ref{lem:ext} below.)

Finally, using the goodness of the free group, this reduces to
\[E_2^{II}=H^i(P_n,H^j(F_n,\QQ))=E_2^{I} \; ,\]
hence proving the claim.
\end{proof}

\begin{lem}{\cite[Theorem 3.1]{falkrandell}}\label{lem:falkrandell}
Let 
\[1\to F\to G\to H\to 1\]
be a split short exact sequence of groups such that the action of $H$ on the abelianization of $F$ is trivial. Then, for each $k$, there is a short exact sequence
\[1\to F/\Gamma^kF\to G/\Gamma^kG\to H/\Gamma^kH\to 1 \; .\]
\end{lem}

\begin{proof}
Since this result is not quite phrased in this manner in \cite{falkrandell}, we provide some details. We prove this by induction on $k$. The base case holds by assumption. We may translate the question into a lemma about fiber sequences. Consider the following commutative diagram of spaces
\[
\xymatrix{
B\Gamma^{k}F/\Gamma^{k+1}F\ar[r]\ar[d]& B\Gamma^kG/\Gamma^{k+1}G\ar[r]\ar[d]&B\Gamma^kH/\Gamma^{k+1}H\ar[d]\\ 
BF/\Gamma^{k+1}F\ar[r]\ar[d]& BG/\Gamma^{k+1}G\ar[r]\ar[d]&BH/\Gamma^{k+1}H\ar[d]\\ 
BF/\Gamma^{k}F\ar[r]& BG/\Gamma^{k}G\ar[r]&BH/\Gamma^{k}H
}
\]
The bottom row is a fiber sequence by the induction hypothesis, the top row is a fiber sequence by \cite[Theorem 3.1]{falkrandell}. The three columns are fiber sequences. Moreover, the inclusion $\Gamma^kF/\Gamma^{k+1}F\to F/\Gamma^{k+1}F$ is central. This implies that the leftmost column has trivial monodromy and so is a principal fibration. Therefore, the leftmost column extends to a fiber sequence
\[BF/\Gamma^{k+1}F\to BF/\Gamma^{k}F\to K(\Gamma^kF/\Gamma^{k+1}F,2)\]
and similarly for the other two columns. Since fiber sequences are stable under the operation of taking fibers, we are done.
\end{proof}

\begin{lem}\label{lem:ext}
Suppose we have a central extension
\[
1 \to A \to B \to C \to 1
\]
acted on by a group $G$, and the action of $G$ on $H^*(A, \QQ)$ and $H^*(C, \QQ)$ is $\QQ$-unipotent. Then the action of $G$ on $H^*(B, \QQ)$ is $\QQ$-unipotent.
\end{lem}
\begin{proof}
Under the assumptions, the $E_2$-page of the Leray-Serre spectral sequence is of the form $H^*(C, \QQ) \otimes H^*(A, \QQ)$, and is therefore a $\QQ$-unipotent representation of $G$. Since $\QQ$-unipotent representations are stable under quotients, subobjects and extensions, the $E_{\infty}$-page is $\QQ$-unipotent and so is $H^*(B, \QQ)$.
\end{proof}

\section{Chains on pro-spaces}

The purpose of this appendix is to give a proof of Proposition \ref{prop:chains-dend}. Throughout, $\kk$ denotes $\FF_p$ or $\QQ$. We will use the language of $\infty$-categories freely. 

\medskip
In this section, we take a different perspective on $\infty$-operads. The analogy with the strict case is the following. 
Given a category $\cat{C}$ with finite products, let $\cat{C}^\times$ denote the operad whose colors are collections of objects in $\cat{C}$ and operations as follows: an element in $\cat{C}^\times(c_1, \dots, c_n; c)$ is a morphism $c_1 \times \dots \times c_n \to c$ in $\mathbf{C}$. Let $\mathcal{O}$ denote the colored operad which encodes monochromatic symmetric operads. Somewhat tautologically, a monochromatic operad in $\cat{C}$ is the same data as a map $\mathcal{O} \to \cat{C}^{\times}$ of operads.

\medskip
Similarly, for an $\infty$-category $\cat{C}$ with finite products, let us denote by $\cat{Alg}_{\mathcal{O}}(\cat{C}^\times)$ the $\infty$-category of $\mathcal{O}$-algebras in $\cat{C}^\times$, by which we mean the $\infty$-category of maps of $\infty$-operads $\mathcal{O}\to\cat{C}^\times$. There is a comparison map
\begin{equation}\label{eq:comparison-operads}
\Op_{\infty}\cat{C} \to \cat{Alg}_{\mathcal{O}}(\cat{C}^\times)
\end{equation}
of $\infty$-categories whose construction we explain below. In this context, an object in $\Op_{\infty}\cat{C}$ is an $\infty$-functor $\mathsf{Tree}^\op \to \cat{C}$ satisfying a Segal condition and $X_\eta \sim *$.

\medskip
Let $\Op\cat{Sets}$ denote the category of \emph{monochromatic} operads in $\cat{Sets}$. Consider the functor 
\[
\iota : \mathsf{Tree} \to \Op\cat{Sets}
\]
that sends a tree $T$ to the coproduct, taken over the vertices $v$ of $T$, of the free (monochromatic) operad on a single operation in arity $|v|$, the number of inputs at $v$. Let $\mathcal{F}$ be the essential image of $\iota^\op$, the opposite of $\iota$. Then there is a fairly tautological map
\[
\Op_{\infty}\cat{C} \to \mathrm{Fun}^{\times}(\mathcal{F}, \cat{C})
\]
where $\mathrm{Fun}^{\times}$ stands for product-preserving $\infty$-functors. Moreover, the assignment that takes an $\infty$-category with finite products $\cat{C}$ to an $\infty$-operad $\cat{C}^\times$ defines a functor
\[
\mathrm{Fun}^{\times}(\mathcal{F}, \cat{C}) \to \cat{Alg}_{\mathcal{F}^\times}(\cat{C}^\times) \;
\]
where the target is, as before, the $\infty$-category of $\infty$-operad maps $\mathcal{F}^\times\to\cat{C}^{\times}$. To complete the description of (\ref{eq:comparison-operads}), we pre-compose along a map $\mathcal{O} \to \mathcal{F}^\times$. To specify one such map amounts to producing an operad in $\mathcal{F}$ or, alternatively, a cooperad in $\mathcal{F}^\op = \iota(\mathsf{Tree})$. There is a canonical choice for such, namely the cooperad whose $n$-term is the free operad on a single $n$-arity operation. 

This defines the map (\ref{eq:comparison-operads}). Each of the intermediary maps is in fact an equivalence of $\infty$-categories, but that does not concern us here. For more details, we refer the reader for example to \cite{chuenriched} or \cite{luriehigher}.

\medskip
In order to prove Proposition \ref{prop:chains-dend}, we also need to deal with pro-chain complexes and $\infty$-operads of such.

\begin{defn}
The $n$-cohomology group of a pro chain complex $C=\{C_i\}_{i\in I}$ is defined by the formula
\[H^k(C)=\colimsub{i\in I}H_{-k}(C_i^\vee)\]
where $\vee$ denotes the linear dual. 
\end{defn}
Let $\Pro(\cat{Ch}_*)_{\kk}$ be the localization of $\Pro(\cat{Ch}_*)$ with respect to the cohomology equivalences. One way to show that this exists as an $\infty$-category is to use \cite[Theorem 4.4]{christensenisaksen} in order to construct a model category that presents it. (The model category $\cat{Ch}_*$ is not simplicial but the authors add in \cite[Remark 2.7]{christensenisaksen} that this assumption is not crucial.) 

Let $L_\kk$ denote the localization functor $\Pro(\cat{Ch}_*) \to \Pro(\cat{Ch}_*)_\kk$. The $\infty$-category $\Pro(\cat{Ch}_*)_\kk$ can be identified with the pro-category of the category of chain complexes with finite dimensional homology. As in the case of spaces, there exists an adjunction
\[\cat{Ch}_*\rightleftarrows\Pro(\cat{Ch_*})_\kk\]
whose left adjoint is denoted $C\mapsto C^{\wedge}$, it is simply the composition of the inclusion
\[\cat{Ch}_*\to\Pro(\cat{Ch}_*)\]
with the localization functor $L_\kk$. The right adjoint is denoted $C\mapsto \Mat(C)$. Since we do everything $\infty$-categorically in this section, we allow ourselves to write $\Mat$ instead of $\RR\Mat$. The functor $\Mat$ simply takes a pro-chain complex to its limit. 

Also, taking singular chains with coefficients in $\kk$ gives rise to a left adjoint functor:
\[C_*:\ProS_\kk\to\Pro(\cat{Ch}_*)_\kk \; ,\]
which takes a pro-space $X=\{X_i\}_{i\in I}$ to the pro chain complex $L_\kk \{C_*(X_i)\}_{i\in I}$, that is, the $\kk$-cohomological localization of  the pro-chain complex $\{C_*(X_i)\}_{i\in I}$ obtained by taking singular chains objectwise. The right adjoint of this functor has an easy description. If $C_*$ is a chain complex with finite dimensional homology, then the space $\map(\Bbbk,C_*)$ lives in $\S^{cofin}_\kk$. This defines the right adjoint on constant pro-objects. By the universal property of the pro-category, we can then extend this uniquely into a limit preserving functor
\[\Pro(\cat{Ch}_*)_\kk\to\ProS_\kk \; .\]

\begin{lem}\label{lem:chains completion}
Let $X$ be a space of finite type $\kk$-type (i.e. whose homology with coefficients in $\kk$ is finite dimensional in each degree). Then the map
\[
C_*(X) \to \Mat(C_*(X^\wedge))
\]
adjoint to $C_*(X)^\wedge \to C_*(X^\wedge)$, is a weak equivalence of chain complexes.
\end{lem}

\begin{proof}
This map admits a factorization through the unit map, as follows:
\[
C_*(X) \to \Mat(C_*(X)^{\wedge}) \to \Mat(C_*(X^\wedge)) \; .
\]
The right-hand map is a weak equivalence in $\Pro(\cat{Ch}_*)_\kk$. This follows from the commutativity of the diagram of left adjoint functors
\[
	\begin{tikzpicture}[descr/.style={fill=white}]
	\matrix(m)[matrix of math nodes, row sep=2.5em, column sep=2.5em,
	text height=1.5ex, text depth=0.25ex]
	{
	\S & \ProS_\kk  \\
	\cat{Ch}_* & \Pro(\cat{Ch}_*)_{\kk} \;  \\
	};
	\path[->,font=\scriptsize]
		(m-1-1) edge node [auto] {} (m-1-2)
		(m-1-1) edge node [auto] {} (m-2-1)
		(m-1-2) edge node [auto] {} (m-2-2)
		(m-2-1) edge node [auto] {} (m-2-2);
	\end{tikzpicture}
\]
which itself follows from the obvious commutativity of the corresponding diagram of right adjoints. 

For the left-hand map, we can use as a model for the completion of $C_*(X)$ its Postnikov tower viewed as a pro-object in chain complexes. Indeed, viewing $C_*(X)$ as a constant pro-object, the map from $C_*(X)$ to its Postnikov tower is an equivalence in $\Pro(\cat{Ch}_*)_{\kk}$. The finite type assumption guarantees that the Postnikov tower is a local object in $\Pro(\cat{Ch}_*)_\kk$. Since the map from $C_*(X)$ to the limit of the Postnikov tower is a weak equivalence, we are done.
\end{proof}

\begin{proof}[Proof of Proposition \ref{prop:chains-dend}]
The functor
\[
C^\prime_* : \cat{Op}_{\infty}\ProS_\kk \to \cat{OpCh_*} \;
\]
is constructed as the composition of four functors. The first one is the functor
\[\cat{Op}_{\infty}\ProS_\kk\to\cat{Alg}_{\mathcal{O}}(\ProS_\kk^\times)\]
that we described at the beginning of the section. The second is the functor
\[\cat{Alg}_{\mathcal{O}}(\ProS_\kk^\times)\to\cat{Alg}_{\mathcal{O}}(\Pro(\cat{Ch}_*)_\kk^\otimes)\]
induced by $C_*$. The third functor is the functor
\[\cat{Alg}_{\mathcal{O}}(\Pro(\cat{Ch}_*)_\kk^\otimes)\to \cat{Alg}_{\mathcal{O}}(\cat{Ch}_*^\otimes)\]
induced by $\Mat$. This is justified since the functor 
\[
\Mat:\Pro(\cat{Ch}_*)^{\otimes}_\kk\to\cat{Ch}_*^{\otimes}
\]
is lax symmetric monoidal by general principles \cite[Corollary 7.3.2.7]{luriehigher}, as it is the right adjoint to the strong symmetric monoidal functor $(-)^\wedge$. 

Finally, by \cite{Hinich}, there is a rectification functor from $\cat{Alg}_{\mathcal{O}}(\cat{Ch}_*^\otimes)$ to the $\infty$-category underlying the model of (strict, monochromatic) dg-operads $\cat{OpCh_*}$. The rectification of $\Mat(C_*(X))$ is, by definition, $C^\prime_*(X)$.

Now, take an operad $P$ in spaces and view it as a map $\mathcal{O} \to \S^\times$ of $\infty$-operads. Completing we obtain $P^\wedge : \mathcal{O} \to \ProS^\times_\kk$ (by Proposition \ref{prop : completion commutes with product}). There is a canonical functor
\[
C_*(P) \to \Mat(C_*(P^\wedge))
\]
which is adjoint to $C_*(P)^\wedge \to C_*(P^\wedge)$. This defines the required map 
\[
C_*(P) \to C_*^\prime(P)
\]
since $C^\prime_*(P)$ is the result of rectifying $\Mat(C_*(P^\wedge))$. That this map is an equivalence when each $P(n)$ is of finite type is an immediate consequence of Lemma \ref{lem:chains completion}.
\end{proof}

\bibliographystyle{acm}

\bibliography{biblio}

\end{document}